\documentclass[final,12pt]{colt2023} %

\usepackage{mathtools}
\usepackage{bm}

\usepackage{thmtools,thm-restate, amssymb,amsmath, mathtools}
\usepackage{mdframed}
\usepackage{algorithm}
\usepackage{algorithmic}

\usepackage{dblfloatfix}
\usepackage{graphicx}
\usepackage{cleveref}
\usepackage{float}
\usepackage[makeroom]{cancel}
\usepackage{makecell}
\usepackage{setspace}
\usepackage{enumitem}
\usepackage{booktabs}
\usepackage{color, colortbl}
\usepackage{xcolor}

\usepackage[rightcaption]{sidecap}
\usepackage[font=small]{caption}
\usepackage{times}

\definecolor{Gray}{gray}{0.9}
\definecolor{TableRowColor}{rgb}{0.91,1.0,0.98}
\definecolor{TableMentionColor}{rgb}{0.0, 0.47, 0.44}

\newcommand{\ra}[1]{\renewcommand{\arraystretch}{#1}}

\def\argmin{\mathop{\rm argmin}}
\def\Argmin{\mathop{\rm Argmin}}
\def\la{\langle}
\def\ra{\rangle}
\def\beq{\begin{equation}}
\def\eeq{\end{equation}}
\def\ba{\begin{array}}
	\def\ea{\end{array}}
\newcommand{\R}{\mathbb{R}}
\newcommand{\Def}{\stackrel{\mathrm{def}}{=}}

\newcommand{\mat}[1]{\mathbf{#1}}

\newtheorem{assumption}{Assumption}

\let\origtheassumption\theassumption

\DeclareMathAlphabet{\mathcalorigin}{OMS}{cmsy}{m}{n}

\renewcommand{\aa}{\mathbf{a}}
\renewcommand{\AA}{\mathbf{A}}
\providecommand{\bb}{\mathbf{b}}

\providecommand{\DD}{\mathbf{D}}
\providecommand{\ee}{\mathbf{e}}
\providecommand{\ff}{\mathbf{f}}

\providecommand{\hh}{\mathbf{h}}

\providecommand{\nn}{\mathbf{n}}

\providecommand{\QQ}{\mathbf{Q}}

\renewcommand{\ss}{\mathbf{s}}

\providecommand{\uu}{\mathbf{u}}
\providecommand{\vv}{\mathbf{v}}

\providecommand{\xx}{\mathbf{x}}
\providecommand{\XX}{\mathbf{X}}
\providecommand{\yy}{\mathbf{y}}
\providecommand{\zz}{\mathbf{z}}

\providecommand{\cO}{\mathcal{O}}
\newcommand{\xopt}{\xx^{\star}}

\newcommand{\X}{\mathcalorigin{X}}
\newcommand{\diam}{\mathcalorigin{D}_{\X}}

\newcommand{\bigO}[1]{\mathcalorigin{O}\left(#1\right)}
\newcommand{\bigOtilde}[1]{\tilde{\mathcalorigin{O}}\left(#1\right)}

\providecommand{\blambda}{\bm{\lambda}}

\def\nn{{ \nonumber }}

\renewcommand{\epsilon}{\varepsilon}

\newcommand{\norm}[1]{\left\lVert \, #1 \, \right\rVert}

\newcommand{\normsqr}[1]{\norm{#1}^2}

\declaretheorem[name=Theorem,numberwithin=section]{thm-rest}
\declaretheorem[name=Lemma,numberwithin=section]{lem-rest}
\declaretheorem[name=Corollary,numberwithin=section]{cor-rest}
\declaretheorem[name=Definition,numberwithin=section]{def-rest}
\declaretheorem[name=Proposition,numberwithin=section]{prop-rest}
\declaretheorem[name=Assumption,numberwithin=section]{ass-rest}

\title[Linearization Algorithms for Fully Composite Optimization]{Linearization Algorithms for Fully Composite Optimization}

\coltauthor{%
 \Name{Maria-Luiza Vladarean} \Email{maria-luiza.vladarean@epfl.ch}\\
 \addr \'Ecole Polytechnique F\'ed\'erale de Lausanne
 \AND
 \Name{Nikita Doikov} \Email{nikita.doikov@epfl.ch}\\
 \addr \'Ecole Polytechnique F\'ed\'erale de Lausanne
  \AND
 \Name{Martin Jaggi} \Email{martin.jaggi@epfl.ch}\\
 \addr \'Ecole Polytechnique F\'ed\'erale de Lausanne
  \AND
 \Name{Nicolas Flammarion} \Email{nicolas.flammarion@epfl.ch}\\
 \addr \'Ecole Polytechnique F\'ed\'erale de Lausanne
}

\begin{document}

\maketitle

\begin{abstract}%
  This paper studies first-order algorithms for solving fully composite optimization problems over convex and compact sets. We leverage the structure of the objective by handling its differentiable and non-differentiable components separately, linearizing only the smooth parts. This provides us with new generalizations of the classical Frank-Wolfe method and the Conditional Gradient Sliding algorithm, that cater to a subclass of non-differentiable problems. Our algorithms rely on a stronger version of the linear minimization oracle, which can be efficiently implemented in several practical applications. We provide the basic version of our method with an affine-invariant analysis and prove global convergence rates for both convex and non-convex objectives. Furthermore, in the convex case, we propose an accelerated method with correspondingly improved complexity. Finally, we provide illustrative experiments to support our theoretical results.
\end{abstract}


\section{Introduction}
In this paper we consider fully composite optimization problems of the form 
\begin{equation}
	\label{IntroMainPb}
	\min\limits_{\xx \in \X} \left[ \varphi(\xx) \;\; \Def \;\;  F( \ff(\xx), \xx)\right],
\end{equation}
where $\X$ is a convex and compact set, $F: \R^n \times \X \to \R$ is a simple but possibly \textit{non-differentiable} convex function 
and $\ff: \X \to \R^n$ is a smooth mapping,
which is the \emph{main source of computational burden}.

Problems of this type cover and generalize many classical use-cases of composite optimization and are often encountered in applications.
In this work, we develop efficient algorithms for solving \eqref{IntroMainPb} by leveraging the \emph{structure of the objective} and using the \textit{linearization principle}. Our method generalizes the well-known Frank-Wolfe algorithm \citep{frank1956algorithm} and ensures asymptotically faster convergence rates compared to methods treating $\varphi$ in a black-box fashion.

A classical algorithm for solving smooth optimization problems is the Gradient Descent method (GD),
proposed by Cauchy in 1847 (see historical note by \cite{lemarechal2012cauchy}). It rests on the idea of linearizing the function around the current iterate, taking a step in the negative gradient direction and projecting the result onto the feasible set $\X$ for $k \geq 0$:\vspace{-2mm}
\beq \label{GradMethod}
\ba{rcl}
\yy_{k + 1} & = & 
\pi_{\X}\bigl( \yy_k - \alpha_k \nabla \varphi(\yy_k)\bigr), \qquad \alpha_k > 0,
\ea
\eeq
where $\pi_{\X}$ is the projection operator onto $\X$. Surprisingly, the same kind of iterations can minimize \emph{general} non-smooth convex functions by substituting $\nabla \varphi(\yy_k)$ with any \emph{subgradient} in the subdifferential $\partial \varphi(\yy_k)$. The resulting Subgradient method was proposed by \cite{shor1985minimization}.

Another notable example for smooth optimization over a convex and  \emph{bounded} constraint set $\X$ is the Frank-Wolfe (FW) method~\citep{frank1956algorithm}. Again, a linearization of the objective around the current iterate is used to query the so-called \emph{linear minimization oracle} (LMO) associated with $\X$, for every $k \geq 0$:
\beq \label{FWMethod}
\ba{rcl}
\yy_{k + 1} & \in & \Argmin\limits_{\xx} \bigl\{ \, \la \nabla \varphi(\yy_k), \xx \ra
\; : \;  \xx \in \yy_k + \gamma_k( \X - \yy_k ) \, \bigr\}, \qquad \gamma_k \in (0, 1].
\ea
\eeq
Steps of type~\eqref{FWMethod} can be significantly cheaper than those involving projections~\eqref{GradMethod} for a few notable domains such as nuclear norm balls and spectrahedrons~\citep{combettes2021complexity}, making FW the algorithm of choice in such scenarios. 
Moreover, the solutions found by FW methods can benefit from additional properties such as sparsity~\citep{jaggi2013revisiting}. These desirable features make FW methods suitable for large scale optimization, a fact which led to an increased interest in recent years (we point the reader to the monograph of~\citet{braun2022FWbook} for a detailed presentation). Unfortunately, the vanilla FW algorithm does not extend to non-differentiable problems in the same straightforward manner as GD -- a counterexample is given by~\citet{nesterov2018complexity}. The question of developing non-smooth versions of the FW algorithm therefore remains open, and is the main focus of this article. 

Finally, we touch on the issue of convergence rates -- a principal means of theoretically characterizing optimization algorithms. The classical monograph of \cite{nemirovskii1983problem} establishes that the $\cO(1 / \sqrt{k})$ rate of the Subgradient method is optimal for \emph{general non-differentiable} convex problems, while the $\cO(1 / k)$ rate of its counterpart GD is far from the lower bound of $\Omega(1/k^2)$ for $L-$smooth convex functions. Similar results are established by~\citet{lan2013complexity} for LMO-based algorithms, although in this case the $\cO(1 / k)$ rate is matched by a lower bound for smooth convex minimization. This relatively slow convergence of FW algorithms is a result of their  \textit{affine-invariant} oracle, which is independent on the choice of norm. In light of these lower bounds, one can only hope to improve  convergence rates by imposing additional structure on the problem to be solved.

The present work leverages this observation and studies a subclass of (possibly) non-smooth and non-convex problems with the \emph{specific structure} of~\eqref{IntroMainPb}. Our methods require only linearizations of the differentiable component $\ff$, while the non-differentiable function $F$ is kept as a part of the subproblem solved within oracle calls. We show that this approach
is a viable way of generalizing FW methods to address problem~\eqref{IntroMainPb}, with the possibility of acceleration. Our contributions can be summarized as follows.
\begin{itemize}[itemsep=0pt, topsep=1pt]
	\item We propose a basic method for problem~\eqref{IntroMainPb}, which is \textit{affine-invariant} and equipped with accuracy certificates. We prove the global convergence rate 
	of $\cO(1/k)$ in the convex setting, and of $\tilde{\cO}(1 / \sqrt{k})$ in the non-convex case. 
	\item We propose an accelerated method with inexact proximal steps which attains a convergence rate of $\cO(1 / k^2)$ for convex problems. Our algorithm achieves the optimal $\bigO{\epsilon^{-1/2}}$ oracle complexity for smooth convex problems in terms of the number of computations of $\nabla \ff$.
	\item We provide proof-of-concept numerical experiments, that demonstrate the efficiency of our approach for solving composite problems.
\end{itemize}

\paragraph{Related Work.}
The present work lies at the intersection of two broad lines of study: general methods for composite optimization and FW algorithms. 
The former category encompasses many approaches that single out non-differentiable components in the objective's structure, and leverage this knowledge in the design of efficient optimization algorithms.
This approach originated in the works 
of \citet{burke1985descent,burke1987second,nesterov1989effective,nemirovski1995information,pennanen1999graph,boct2007new,boct2008new}.
A popular class of \textit{additive} composite optimization problems
was proposed by \citet{beck2009fast,nesterov2013gradient}
and the modern algorithms for general composite formulations were developed by
\citet{cui2018composite,drusvyatskiy2018error,drusvyatskiy2019efficiency,bolte2020multiproximal,burke2021study,doikov2022high}.

The primitive on which most of the aforementioned methods rely is a \textit{proximal-type} step -- a generalization of~\eqref{GradMethod}. Depending on the geometry of the set $\X$, such steps may pose a significant computational burden. \citet{doikov2022high} propose an alternative \emph{contracting-type} method for \textit{fully composite} problems, 
which generalizes the vanilla FW algorithm. Their method relies on a simpler primitive built on the linearization principle, which can be much cheaper in practice.
We study the same problem structure as \citet{doikov2022high} and devise methods with several advantages over the aforementioned approach, including an \textit{affine-invariant analysis}, \textit{accuracy certificates}, convergence guarantees for non-convex problems and, in the convex case, an accelerated convergence. Moreover, we decouple stepsize selection from the computational primitive, to enable efficient line search procedures.%

Our methods are also intimately related to FW algorithms, which they generalize. For smooth and convex problems, vanilla FW converges at the cost of $\bigO{\epsilon^{-1}}$ LMO and \emph{first order oracle} (FO) calls in terms the Frank-Wolfe gap -- an accuracy measure bounding functional suboptimality~\citep{jaggi2013revisiting}. For smooth non-convex problems, a gap value of at most $\epsilon$ is attained after $\bigO{\epsilon^{-2}}$ LMO and FO calls~\citep{lacoste2016convergence}. Due to the relatively slow convergence of LMO-based methods, recent efforts have gone into devising variants with improved guarantees. The number of FO calls was reduced to the lower bound for smooth convex optimization by~\citet{lan2016conditional}, local acceleration was achieved following a burn-in phase by~\citet{diakonikolas2020locally, carderera2021parameter, chen2022accelerating}, and empirical performance was enhanced by adjusting the update direction with gradient information by~\citet{combettes2020boosting}. Of the aforementioned works, closest to ours is the Conditional Gradient Sliding (CGS) algorithm proposed by~\citet{lan2016conditional} and further studied by~\citet{yurtsever2019conditional, qu2018non}. CGS uses the acceleration framework of~\citet{nesterov1983method} and solves the projection subproblem inexactly via the FW method, achieving the optimal complexity of $\bigO{\epsilon^{-1/2}}$ FO calls for smooth convex problems. We rely on a similar scheme for improving FO complexity in the convex case.

In the context of generic non-smooth convex objectives, the FW algorithm was studied by~\citet{lan2013complexity}, who proposes a smoothing-based approach matching the lower bound of $\Omega(\epsilon^{-2})$ LMO calls. The method however requires $\bigO{\epsilon^{-4}}$ FO calls, a complexity which is later improved to $\bigO{\epsilon^{-2}}$ by \citet{garber2016linearly} through a modified LMO for polytopes, by~\citet{ravi2019deterministic} with a (differently) modified LMO, and finally by \citet{thekumparampil2020projection} through a combination of smoothing and the CGS algorithm. Our algorithm, instead, leverages the structure of problem~\eqref{IntroMainPb} and a modified LMO to achieve improved rates, with the added benefit of an affine invariant method and analysis. We also mention FW methods for additive composite optimization~\citep{argyriou2014hybrid, yurtsever2018conditional, yurtsever2019conditional, zhao2022analysis}, with the former three relying on proximal steps and the latter assuming a very restricted class of objectives.

Finally, two concurrent works study FW methods for some restricted classes of non-smooth and non-convex problems. \citet{de2023short} shows that vanilla FW with line-search can be applied to the special class of upper$-C^{1, \alpha}$ functions, when one replaces gradients with an arbitrary element in the Clarke subdifferential. A rate of $\bigO{\epsilon^{-2}}$ is shown for reaching a Clarke-stationary point in a setting comparable to ours. A similar rate is shown by~\citet{kreimeier2023frank} for reaching a $d$-stationary point of abs-smooth functions through the use of a modified LMO. Both these algorithms are structure-agnostic. A summary of method complexities for solving non-smooth problems is provided in Table~\ref{table-complexities}. 

\begin{table}[t!]
	\footnotesize
	\centering
	\renewcommand{\arraystretch}{1.2}
	\setlength{\tabcolsep}{1.6pt}
	\begin{tabular}{ll@{\hspace{0pt}}ccc@{\hspace{-2.5pt}}r}\toprule
		\textbf{\small Reference} &  \textbf{\small $\varphi$ subclass} & \textbf{\small Use structure?} &\textbf{\small \# FO}  & \textbf{\small \# PO/LMO} & \textbf{\small Observations}  \\
		\midrule
		\scriptsize \citet{shor1985minimization} & cvx, L-cont & no & $\bigO{\epsilon^{-2}}$$\color{TableMentionColor}{^{\text{\tiny(1)}}}$ &  $\bigO{\epsilon^{-2}}$$\color{TableMentionColor}{^{\text{\tiny(1)}}}$ &  {\scriptsize projection} \\
		\scriptsize\citet{thekumparampil2020projection} & cvx, L-cont & no &  $\bigO{\epsilon^{-2}}$$\color{TableMentionColor}{^{\text{\tiny(1)}}}$ & $\bigO{\epsilon^{-2}}$$\color{TableMentionColor}{^{\text{\tiny(1)}}}$ & { \scriptsize smoothing, vanilla LMO}\\
		\scriptsize\citet{doikov2022high} & cvx, fully-comp & yes & $\bigO{\epsilon^{-1}}$$\color{TableMentionColor}{^{\text{\tiny(1)}}}$ &  $\bigO{\epsilon^{-1}}$$\color{TableMentionColor}{^{\text{\tiny(1)}}}$& {\scriptsize modif. LMO} \\
		\rowcolor{TableRowColor}\textbf{\scriptsize (this work) Alg. 2} & cvx, fully-comp. &yes & $\bigO{\epsilon^{-1/2}}$$\color{TableMentionColor}{^{\text{\tiny(1)}}}$ &  $\bigO{\epsilon^{-1}}$$\color{TableMentionColor}{^{\text{\tiny(1)}}}$ & {\scriptsize {modif. LMO}} \\
		\midrule\\[-3mm]
		\scriptsize\citet{de2023short} & non-cvx, upper-$C^{1,\alpha}$ & no & $\bigO{\epsilon^{-2}}$ $\color{TableMentionColor}{^{\text{\tiny(2)}}}$ & $\bigO{\epsilon^{-2}}$ $\color{TableMentionColor}{^{\text{\tiny(2)}}}$ & {\scriptsize vanilla LMO}\\
		\scriptsize\citet{kreimeier2023frank} & non-cvx, abs-smooth & no & $\bigO{\epsilon^{-2}}$ $\color{TableMentionColor}{^{\text{\tiny(3)}}}$& $\bigO{\epsilon^{-2}}$ $\color{TableMentionColor}{^{\text{\tiny(3)}}}$& {\scriptsize modif. LMO}   \\
		\scriptsize\citet{drusvyatskiy2019efficiency} & non-cvx, comp & yes & $\bigO{\epsilon^{-2}}$  $\color{TableMentionColor}{^{\text{\tiny(4)}}}$ & $\bigO{\epsilon^{-2}}$ $\color{TableMentionColor}{^{\text{\tiny(4)}}}$ & {\scriptsize prox. steps} \\
		\rowcolor{TableRowColor}\textbf{\scriptsize (this work) Alg. 1} & non-cvx, fully-comp &yes & $\bigOtilde{\epsilon^{-2}}$ $\color{TableMentionColor}{^{\text{\tiny(5)}}}$ & $\bigOtilde{\epsilon^{-2}}$ $\color{TableMentionColor}{^{\text{\tiny(5)}}}$& { \scriptsize modif. LMO}\\
		\bottomrule
	\end{tabular}
	\caption{ Summary of convergence complexities for solving non-smooth problems. Note $\color{TableMentionColor}{^{\text{(1)}}}$ marks complexities reaching an $\epsilon$ functional residual. Note $\color{TableMentionColor}{^{\text{(2)}}}$ marks complexities for reaching Clarke-stationary points. Note $\color{TableMentionColor}{^{\text{(3)}}}$ marks complexities for obtaining $d-$stationary points. Note $\color{TableMentionColor}{^{\text{(4)}}}$ 
		marks the complexity for reaching a small norm of the gradient mapping. Finally, note $\color{TableMentionColor}{^{\text{(5)}}}$ marks the complexity of minimizing the positive quantity~\eqref{gapCertificate}.}
	\label{table-complexities}
\end{table}

\paragraph*{Notation.}
We denote by $[n]$ the set $\{1, \ldots n\}$ and by $\norm{\cdot}$ the standard Euclidean norm, unless explicitly stated otherwise. We define the diameter of a bounded set $\X$ as $\diam \Def \max_{\zz, \yy \in \X} \{\norm{\zz - \yy}\}$. We use the notation $\Delta_n \Def \{ \blambda \in \R^{n}_{+} \, : \, \la \blambda, \ee \ra = 1 \}$ to denote the standard $n$-dimensional simplex, where $\ee$ is the vector of all ones. For a differentiable, scalar-valued function $f: \R^{d} \to \R$ we use $\nabla f(\xx) \in \R^d$ to denote its gradient vector and $\nabla^2 f(\xx) \in \R^{d \times d}$ to denote its Hessian matrix. For a differentiable vector-valued function $\ff: \R^{d} \to \R^{n}$ defined as $\ff = (f_1, f_2, \ldots f_n)$ we denote by $\nabla \ff(\xx)$ its Jacobian matrix defined as $ \nabla \ff(\xx) =  \sum_{i = 1}^n \ee_i \nabla f_i(\xx)^{\top}  \in \R^{n \times d},$
where $\ee_i$ are the standard basis vectors in $\R^n$.  We represent the second directional derivatives applied to the same direction $\hh \in \R^d$  as
$
\nabla^2 f(\xx)[\hh]^2 \Def  \la \nabla^2 f(\xx) \hh, \hh \ra \in  \R$, and $ 
\nabla^2 \ff(\xx)[\hh]^2 \Def \sum_{i = 1}^n \ee_i \nabla^2 f_i(\xx)[\hh]^2 \in \R^n$. 
\section{Problem Setup, Assumptions and Examples}
\label{SectionAssumptions}
The problems addressed by this work are represented by the following structured objective
\beq \label{MainProblem}
\ba{rcl}
\varphi^{\star} & = & \min\limits_{\xx \in \X} 
\Bigl[\,
\varphi(\xx) \;\; \Def \;\;  F( \ff(\xx), \xx)
\,
\Bigr], \qquad \X \subset \R^d,
\ea
\eeq
where $\X$ is a convex and compact set and the inner mapping $\ff: \X \to \R^n$ is differentiable and defined as $\ff(\xx) = (f_1(\xx), \ldots, f_n(\xx)) \, \in \, \R^n,$ where each $f_i: \X \to \R$ is differentiable. We assume access to a first-order oracle $\nabla \ff$, which is the main source of computational burden. The outer component $F: \R^n \times \X \to \R$, on the other hand, is \emph{directly accessible} to the algorithm designer and is \emph{simple} (see assumptions). However, $F$ is possibly non-differentiable.

In this work, we propose two algorithmic solutions addressing problem \eqref{MainProblem}, which we call a \emph{fully composite problem}. Our methods importantly  assume that subproblems of the form
\beq \label{MainSubproblem}
\ba{rcl}
\Argmin\limits_{\xx \in \X} F\bigl( \mat{A} \xx + \bb, \xx \bigr)
+ \la \uu, \xx \ra
\ea
\eeq
are efficiently solvable, where $\mat{A} \in \R^{n \times d}$ and $\bb \in \R^n, \uu \in \R^d$. 
Oracles of type ~\eqref{MainSubproblem} are sequentially called  during the optimization procedure and take as arguments linearizations of the difficult nonlinear components of~\eqref{MainProblem}. Naturally, solving~\eqref{MainSubproblem} cheaply is possible only when $F$ is simple and $\X$ has an amenable structure. 

In particular, template~\eqref{MainProblem} encompasses to some standard problem formulations. For example, the classical Frank-Wolfe setting is recovered when $F(\uu, \xx) \equiv u^{(1)}$, in which case problem \eqref{MainProblem} becomes $\min_{\xx \in \X} f_1(\xx)$ and subproblem \eqref{MainSubproblem} reduces to a simple LMO: $\argmin_{\xx \in \X} \la \uu, \xx \ra$. The setting of proximal-gradient methods is similarly covered, by letting $F(\uu, \xx) \equiv u^{(1)} + \psi(\xx)$ for a given convex function $\psi$ (e.g., a regularizer). Then, problem~\eqref{MainProblem} reduces to additive composite optimization~$\min_{\xx \in \X} \Bigl\{  f_1(\xx) + \psi(\xx) \Bigr\},$ and subproblem \eqref{MainSubproblem} becomes $\argmin_{\xx \in \X} \Bigl\{  \la \uu, \xx \ra + \psi(\xx) \Bigr\}.$

We now formally state the assumptions on the fully composite problem~\eqref{MainProblem}, along with commentary and examples.

\begin{assumption} \label{AssumptionF}
	The outer function $F: \R^n \times \X \to \R$ is jointly convex in its arguments. Additionally,  $F(\uu, \xx)$ is subhomogeneous in $\uu$:\vspace{-1mm}
	\beq \label{FSubH}
	\ba{rcl}
	F( \gamma \uu, \xx) & \leq & \gamma F(\uu, \xx), 
	\qquad
	\forall \uu \in \R^n, \; \xx \in \X, \; \gamma \geq 1.
	\ea
	\eeq
\end{assumption}

\edef\oldassumption{\the\numexpr\value{assumption}+1}
\setcounter{assumption}{0}
\renewcommand{\theassumption}{\oldassumption\alph{assumption}}

\begin{assumption} \label{AssumptionSmooth}
	The inner mapping $\ff: \X \to \R^n$ is differentiable and the following affine-invariant quantity is bounded:\vspace{-1mm}
	\beq \label{SDef}
	\ba{rcl}
	\mathcal{S} & \!\! = \!\! & \mathcal{S}_{\ff, F, \X}
	\; \Def 
	
	\sup\limits_{\substack{\xx, \yy \in \X, \,\gamma \in (0, 1] \\
			\yy_{\gamma}  = \xx + \gamma(\yy - \xx)}} 
	F\bigl( \frac{2}{\gamma^2}
	\bigl[ \ff(\yy_{\gamma}) - \ff(\xx) - \nabla \ff(\xx)(\yy_{\gamma} - \xx)  \bigr]
	, \yy_{\gamma} \bigr)
	\; < \; +\infty.
	\ea
	\eeq
\end{assumption}

\begin{assumption}\label{AssumptionLipCont} Each component $f_i(\cdot)$
	has a Lipschitz continuous gradient on $\X$ with constant $L_i$:
	\begin{equation*}
		\norm{\nabla f_i(\xx) - \nabla f_i(\yy)} \leq L_i \norm{\xx - \yy} \quad \forall \xx, \yy \in \X, \forall i \in [n].
	\end{equation*}
	We denote the vector of Lipschitz constants by $\mat{L} = (L_1, \ldots, L_n) \in \R^n$.
\end{assumption}
\let\theassumption\origtheassumption

\begin{assumption} \label{AssumptionConvexity}
	Each component $f_i : \X \to \R$ is convex.
	Moreover, $F(\cdot, \xx)$ is monotone $\forall \xx \in \X$. Thus, for any two vectors
	$\uu, \vv \in \R^n$ such that
	$\uu \leq \vv$ (component-wise), it holds that\vspace{-1mm}
	\beq \label{FMonotone}
	\ba{rcl}
	F(\uu, \xx) & \leq & F(\vv, \xx).
	\ea
	\eeq	
\end{assumption}

A few comments are in order. Assumption~\ref{AssumptionF}, which is also required by~\citet{doikov2022high}, represents the formal manner in which we ask that $F$ be simple -- through convexity and bounded growth in $\uu$. This assumption ensures convexity of subproblem~\eqref{MainSubproblem}, irrespective of the nature of $\ff$.

Assumption~\ref{AssumptionSmooth} is a generalization of the standard bounded curvature premise typical for Frank-Wolfe settings~\citep{jaggi2013revisiting}. Requirement~\eqref{SDef} is mild, as it only asks that the curvature of $\ff$ remains bounded under $F$ over $\X$. Importantly, the quantity $\mathcal{S}$ is \textit{affine-invariant} (remains unchanged under affine reparametrizations of $\X$), which enables us to obtain convergence rates with the same property. Further discussion on the importance of affine-invariant analysis for FW algorithms is provided by \citet{jaggi2013revisiting}. For mappings $\ff$ that are twice differentiable, we can bound the quantity $\mathcal{S}$
from Assumption~\ref{AssumptionSmooth} using Taylor's formula
and the second derivatives, as follows
$$
\ba{rcl}
\mathcal{S} & \leq & 
\sup\limits_{\substack{\xx, \yy \in \X, \, \gamma \in [0, 1] \\
		\yy_{\gamma}  = \xx + \gamma(\yy - \xx)}}
F( \nabla^2 \ff(\yy_{\gamma})[\yy - \xx]^2, \xx ).
\ea
$$
This quantity is reminiscent of the quadratic upper-bound used to analyze smooth optimization methods. In particular, for monotone non-decreasing $F$, a compact $\X$ and Lipschitz continuous $\nabla f_i$ with respect to a fixed norm $\| \cdot \|$, the assumption is satisfied with
$$
\ba{rcl}
\mathcal{S} & \leq & F(\mat{L} \diam^2 ) \;\; \Def \;\; \sup\limits_{\xx \in \X} F(\mat{L} \diam^2, \xx ).
\ea
$$ 

Assumption~\ref{AssumptionLipCont} is standard and considered separately from Assumption~\ref{AssumptionSmooth} to allow for different levels of generality in our results. The restriction to $\X$ makes this a locally-Lipschitz gradient assumption on $f_i$.

Finally, Assumption~\ref{AssumptionConvexity}, which is also made by~\citet{doikov2022high}, is required whenever we need to ensure the overall convexity of $\varphi(\xx)$. The monotonicity of $F$ is necessary in addition to convexity of each $f_i$, since the composition of convex functions is not necessarily convex \citep[][]{boyd2004convex}. We rely on this assumption for deriving asymptotically faster convergence rates in the convex setting (Section~\ref{SectionAccelerated}).

To conclude this section, we provide the main application examples that fall under our fully composite template and which satisfy to our assumptions. Further examples can be found in Appendix~\ref{app-extra-examples}.

\begin{example} \label{ExampleMax}
	Let $F(\uu, \xx) \equiv \max\limits_{1 \leq i \leq n} u^{(i)}$. Function $F$ satisfies Assumptions~\ref{AssumptionF}~and~\ref{AssumptionConvexity}
and problem \eqref{MainProblem} becomes
	\beq \label{MinMax}
	\ba{c}
	\min\limits_{\xx \in \X} \max\limits_{1 \leq i \leq n} f_i(\xx),
	\ea
	\eeq
while oracle~\eqref{MainSubproblem} becomes
	\beq \label{MaxSubproblem}
	\ba{rcl}
	\min\limits_{\xx \in \X} \max\limits_{1 \leq i \leq n} \la \aa_i, \xx \ra + b_i
	& \quad \Leftrightarrow \quad &
	\min\limits_{\xx \in \X, t \in \R} \bigl\{ t \; : \; \la \aa_i, \xx \ra + b_i \leq t, \, 1 \leq i \leq n \bigr\}.
	\ea
	\eeq
	Max-type minimization problems of this kind result from scalarization approaches in multi-objective optimization, and their solutions were shown to be (weakly) Pareto optimal~\citep[Chapter~3.1 in][]{miettinen1999nonlinear}. As such, problem~\eqref{MinMax} is relevant to a wide variety of applications requiring optimal trade-offs amongst several objective functions, and appears in areas such as machine learning, science and engineering~\citep[see the introductory sections of, e.g., ][]{daulton2022multi, zhang2020random}. Problem~\eqref{MinMax} also covers some instances of constrained $\ell_{\infty}$ regression.

When $\X$ is a polyhedron, subproblem~\eqref{MaxSubproblem} can be solved via Linear Programming, while for general $\X$ one can resort to  Interior-Point Methods~\citep{nesterov1994interior}. Another option for solving \eqref{MaxSubproblem} is to note that under strong duality~\citep{rockafellar1970convex} we have
	\beq \label{MaxTypeSubproblem}
	\ba{rcl}
	\min\limits_{\xx \in \X} \max\limits_{1 \leq i \leq n} \la \aa_i, \xx \ra + b_i
	& = & 
	\min\limits_{\xx \in \X} \max\limits_{\blambda \in \Delta_n}
	\sum\limits_{i = 1}^n \lambda^{(i)} \bigl[  \la \aa_i, \xx \ra + b_i \bigr]  = 
	\max\limits_{\blambda \in \Delta_n} g(\blambda),
	\ea
	\eeq
	where $g(\blambda) \Def \min\limits_{\xx \in \X} \sum_{i = 1}^n \lambda^{(i)} \bigl[  \la \aa_i, \xx \ra + b_i \bigr] $. The maximization of $g$ in~\eqref{MaxTypeSubproblem} can be done
	very efficiently for small values of $n$ (with, e.g., the Ellipsoid Method or
	the Mirror Descent algorithm), since evaluating $g(\blambda)$ and $\partial g(\blambda)$ reduces to a vanilla LMO call over $\X$. An interesting case is $n = 2$, for which~\eqref{MaxTypeSubproblem} becomes
a \emph{univariate} maximization problem and one may use binary search to solve it at the expense of a logarithmic number of LMOs.
\end{example}

\begin{example} \label{ExampleNorm}
	Let
	$
	F(\uu, \xx) \equiv  \| \uu \| 
	$
	for an arbitrary fixed norm $\| \cdot \|$. Function $F$ satisfies Assumption~\ref{AssumptionF} and problem~\eqref{MainProblem}
	can be interpreted as solving a system of non-linear equations
	over $\X$
	\beq \label{GSProblem}
	\ba{rcl}
	\min\limits_{\xx \in \X} \| \ff (\xx) \|,
	\ea
	\eeq
while oracle~\eqref{MainSubproblem} amounts to solving the (constrained) linear system 
	$
	\min\limits_{\xx \in \X} \| \mat{A}\xx + \bb \|	$. Problems of this kind can be encountered in applications such as robust phase retrieval~\citep{duchi2019solving} with phase constraints.
	
	The iterations of Algorithm~\ref{alg:basic} can be interpreted as a variant
	of the \textit{Gauss-Newton method} \citep{burke1995gauss, nesterov2007modified, tran2020stochastic},
	solving the (constrained) linear systems:
	\beq \label{GSAlgEx}
	\xx_{k + 1}  \in  \Argmin\limits_{\xx \in \X}
	\| \ff(\yy_k) + \nabla \ff(\yy_k)(\xx - \yy_k) \|, \ \ \text{ and }\ \ 
	\yy_{k + 1}  =  (1 - \gamma_k) \yy_k + \gamma_k \xx_{k + 1}.
	\eeq
	In the particular case of solving \textit{systems of non-linear equations} over compact convex sets, our algorithms can be seen as modified Gauss-Newton methods with global convergence guarantees.
\end{example}

\section{The Basic Method}
\label{SectionBasic}
We present the first new method for solving problem~\eqref{MainProblem} in Algorithm~\ref{alg:basic}.
The central idea is to \emph{linearize} the differentiable components of the objective and then to minimize this new model over the constraint $\X$, via calls to an oracle of type~\eqref{MainSubproblem}.
The next iterate is defined as a convex combination with coefficient (or \emph{stepsize}) $\gamma$ between the computed minimizer and the preceding iterate.
\begin{algorithm}[ht!]
	\setstretch{1.1}
	\begin{algorithmic}
		\STATE \textbf{Input:} $\yy_0 \in \X$
		\FOR{$k = 0, 1, \dots $}
		\STATE Compute 
		\vspace{-2mm}
		$$
		\ba{rcl}
		\xx_{k + 1} & \in & \Argmin\limits_{\xx \in \X}
		F\bigl(  \ff(\yy_k) + \nabla \ff(\yy_k) (\xx - \yy_k), \, \xx \bigr)
		\ea
		$$
		\vspace{-3mm}
		\STATE Choose $\gamma_k \in (0, 1]$ by a predefined rule or with line search
		\STATE Set 
		$\yy_{k + 1} = (1 - \gamma_k) \yy_k + \gamma_k \xx_{k + 1}$
		\ENDFOR
	\end{algorithmic}
	\caption{Basic Method}
	\label{alg:basic}
\end{algorithm}

A similar method for tackling problems of type~\eqref{MainProblem} in the convex setting was proposed by \citet{doikov2022high}. 
Different from theirs, our method decouples the parameter $\gamma_k$ from the minimization subproblem.
This change is crucial since it allows us to choose the parameter $\gamma_k$ \emph{after} minimizing the model, thus enabling us to use efficient line search rules. 
Moreover, we provide Algorithm~\ref{alg:basic} with a more advanced \emph{affine-invariant} analysis and establish its convergence in the \emph{non-convex} setup. 

We also mention that for solving problems of type~\eqref{MinMax}, oracle~\eqref{MainSubproblem} reduces to the minimization of a piecewise linear function over $\X$. Therefore, it has the same complexity as the modified LMOs of~\citet{kreimeier2023frank} and, moreover, subproblem~\eqref{MainSubproblem} is convex irrespective of the nature of $\ff$.

\paragraph{Accuracy Certificates.}  The standard \emph{progress metric} of FW algorithms, which Algorithm~\ref{alg:basic} generalizes, is the `Frank-Wolfe gap' or Hearn gap~\citep{hearn1982gap}. For smooth objectives, it is defined as  $\Delta_k = \max_{\yy\in\X} \la \nabla \varphi(\yy_k), \yy_k - \yy \ra$, for each iterate $\yy_k$. This quantity is computed cost-free during the algorithm's iterations and has the desirable property of upper-bounding the suboptimality of the current iterate: $\Delta_k \geq \varphi(\yy_k) - \varphi^{\star}$. 
Notably, its semantics straightforwardly extend to the non-convex setting~\citep{lacoste2016convergence}. Additionally, convergence guarantees on the gap are desirable due to its affine invariance, which aligns with the affine invariance of classical FW algorithm.

Our setting does not permit a direct generalization of the FW gap with all of the above properties. Rather, we introduce the following \emph{accuracy certificate}, which is readily available in each iteration:
\begin{equation}
	\label{gapCertificate}
	\Delta_k \; \Def \;
	\varphi(\yy_k) 
	- F\big( \ff(\yy_k) + \nabla \ff(\yy_k)(\xx_{k + 1} - \yy_k), \xx_{k + 1}\big).
\end{equation}
For minimization of a smooth (not necessarily convex) function, quantity~\eqref{gapCertificate} indeed reduces to the standard FW gap. Moreover, for convex $\varphi(\xx)$ (Assumption~\ref{AssumptionConvexity}) we can conclude that
\beq \label{DeltaKLower}
\ba{rcl}
\Delta_k & \geq & 
\max\limits_{\xx \in \X}\Bigl[ \varphi(\yy_k)
-  F\big(\ff(\yy_k) + \nabla \ff(\yy_k)(\xx - \yy_k), \xx\big) \Bigr] \\[10pt]
& \geq & 
\max\limits_{\xx \in \X}
\Bigl[
\varphi(\yy_k) - F( \ff(\xx), \xx ) \Bigr]
\;\; = \;\; 
\varphi(\yy_k) - \varphi^{\star}.
\ea
\eeq

Hence, for a tolerance $\varepsilon > 0$, the criterion $\Delta_k \leq \varepsilon$ can be used as the stopping condition for our method in convex scenarios. Moreover, the value of $\Delta_k$ can be used for computing the parameter $\gamma_k$ through line search. 

\paragraph{Convergence on Convex Problems.}
In the following, we prove the global convergence of Algorithm~\ref{alg:basic} in case when $\varphi(\xx)$ is convex.

\begin{restatable}{thm-rest}{ThmBasicConvex}
	\label{TheoremBasicConvex}
	Let Assumptions~\ref{AssumptionF}, \ref{AssumptionSmooth},
	and \ref{AssumptionConvexity} be satisfied. 
	Let $\gamma_k \coloneqq \min\{1, \frac{\Delta_k}{\mathcal{S}} \}$
	\textbf{or} $\gamma_k \coloneqq \frac{2}{2 + k}$.
	Then, for $k \geq 1$ it holds that
	\beq \label{ConvergenceConvex}
	\varphi(\yy_k) - \varphi^{\star} \;\; \leq \;\; \frac{2 \mathcal{S}}{1 + k} \;\qquad \;\text { and } \;\qquad\; \min\limits_{1 \leq i \leq k} \Delta_i
	\;\; \leq \;\; 
	\frac{6\mathcal{S}}{k}.
	\eeq
\end{restatable}
Our method recovers the rate of classical FW methods for smooth problems, while being applicable to the wider class of \emph{fully composite problems}~\eqref{MainProblem}. Thus, our $\mathcal{O}(1/k)$ rate improves upon the $\mathcal{O}(1/\sqrt{k})$ of black-box non-smooth optimization. Clearly, the improvement is achievable by leveraging the \emph{structure of the objective} within the algorithm.

\paragraph{Convergence on Non-convex Problems.}
In this case, $\Delta_k$ has different semantics and no longer provides an accuracy certificate for the functional residual. This quantity is nevertheless important, since it enables us to quantify the algorithm's progress in the non-convex setting, while maintaining an affine-invariant analysis. 
The following theorem states the convergence guarantee on $\Delta_k$ for non-convex problems.

\begin{restatable}{thm-rest}{ThmGap}
	\label{TheoremGap}
	Let Assumptions~\ref{AssumptionF} and \ref{AssumptionSmooth} be satisfied.
	Let $\gamma_k \coloneqq \min\{1, \frac{\Delta_k}{\mathcal{S}} \}$
	\textbf{or} $\gamma_k \coloneqq \frac{1}{\sqrt{1 + k}}$.
	Then, for all $k\geq 1$ it holds that 
	\beq \label{ConvergenceGapGeneral}
	\ba{rcl}
	\min\limits_{0 \leq i \leq k} \Delta_i & \leq & 
	\frac{ \varphi(\yy_0) - \varphi^{\star} + 0.5 \mathcal{S}(1 + \ln(k + 1)) }{\sqrt{k + 1}}.
	\ea
	\eeq
\end{restatable}
Theorem~\ref{TheoremGap} recovers a similar rate to the classical FW methods~\citep{lacoste2016convergence}. The line search rule for parameter $\gamma_k$ makes our method universal, thereby allowing us to attain practically faster rates automatically when the iterates lie within a \textit{convex region} of the objective.

As previously mentioned, the progress measure~\eqref{gapCertificate} does not upper-bound functional suboptimality in the general non-convex setting. However, 
in some cases, we may still be able to establish convergence of meaningful quantities 
for non-convex fully composite problems with the linearization method.
Namely, let us consider problem~\eqref{GSProblem} in Example~\ref{ExampleNorm} for the Euclidean norm, i.e., $F(\uu, \xx) = \norm{\uu}$,
and the following simple iterations:
\beq \label{NormBasicIters}
\ba{rcl}
\yy_{k + 1} & \in & \Argmin\limits_{\yy \in \yy_k + \gamma_k(\X - \yy_k)}
\| \ff(\yy_k) + \nabla \ff(\yy_k) (\yy - \yy_k) \|.
\ea
\eeq
 Note that in~\eqref{NormBasicIters}, differently from~\eqref{GSAlgEx}, the value of $\gamma_k$ is selected prior to the oracle call. Denoting the squared objective as
$\Phi(\xx) \Def \frac{1}{2}\bigl[ \varphi(\xx) \bigr]^2 = \frac{1}{2} \| \ff(\xx) \|^2$
and following our analysis, we can state the convergence of process \eqref{NormBasicIters}
in terms of the classical FW gap with respect to $\Phi$.
The proof is deferred to Appendix~\ref{app:proof-prop-3-1}.

\begin{restatable}{prop-rest}{PropNorm}\label{PropNorm}
	Let $\gamma_k \coloneqq \frac{1}{\sqrt{1 + k}}$. Then, 
	for the iterations \eqref{NormBasicIters}, under Assumption~\ref{AssumptionLipCont} and for all $k \geq 1$, it holds that
	$$
	\ba{rcl}
	\min\limits_{0 \leq i \leq k}
	\max\limits_{\yy \in \X }\la \nabla \Phi(\yy_i), \yy_i - \yy \ra
	& \leq & 
	\cO\bigl(\frac{\ln (k)}{\sqrt{k}}\bigr).
	\ea
	$$	
\end{restatable}

We further show in Appendix~\ref{app-section-interp-delta-nonconvex} that $\Delta_k$ can be related to the classical FW gap, when our iterates lie in a smooth region of $F$. Whether we can provide a meaningful interpretation of $\Delta_k$ in the general non-convex case, however, remains an interesting open question.

\section{The Accelerated Method}
\label{SectionAccelerated}

We now move away from the affine-invariant formulation of Algorithm~\ref{alg:basic} to a setting in which, by considering regularized minimization subproblems along with convexity and Lipschitz continuity of gradients, we can \emph{accelerate} the Basic Method. We achieve acceleration by resorting to the well-known three-point scheme of~\citet{nesterov1983method},
in which the proximal subproblem is solved inexactly via calls to oracles of type~\eqref{MainSubproblem}. This approach was first analyzed in the context of FW methods by~\citet{lan2016conditional}.

We propose Algorithm~\ref{alg:fw-fully-comp} which consists of a two-level scheme: an outer-loop computing the values of three iterates $\yy$, $\xx$ and $\zz$ in $\X$, and a subsolver computing inexact solutions to the \textit{proximal subproblem}
\beq \label{ProxPoint}
\ba{rcl}
\Argmin\limits_{\uu \in \X}
\Bigl\{ P(\uu) \Def
F( \ff(\zz) + \nabla \ff(\zz)(\uu - \zz), \uu )
+ \frac{\beta}{2}\| \uu - \xx \|_2^2, \;\; \beta > 0
\Bigr\}.
\ea
\eeq
Note that the minimization in~\eqref{ProxPoint} does not conform to our oracle model~\eqref{MainSubproblem} due to the quadratic regularizer. However, we can approximate its solution by iteratively solving subproblems in which we linearize the squared norm to match the template of \eqref{MainSubproblem}. This procedure, denoted as $\text{InexactProx}$ in Algorithm~\ref{alg:fw-fully-comp}, returns a point $\uu^{+}$ satisfying the optimality condition \emph{$\eta$-inexactly} for some $\eta > 0$:
\begin{align} \label{InexactProxGuarantee}
	&F(\ff(\zz) + \nabla f(\zz)(\uu^{+} - \zz), \uu^{+})
	+ \beta \la \uu^{+} - \xx, \uu^{+}\ra \nn\\[10pt]
	& \hspace{25mm} \leq
	F(\ff(\zz) + \nabla f(\zz)(\uu - \zz), \uu)
	+ \beta\la \uu^{+} - \xx, \uu \ra + \eta, \qquad \forall \uu \in \X.
\end{align}
\begin{algorithm}[ht!]
	\setstretch{1.2}
	\begin{algorithmic}
		\STATE \textbf{Input:} $\yy_0 \in \X$, set $\xx_0 = \yy_0$ 
		\FOR{$k = 0, 1, \dots $}
		\STATE Choose $\gamma_k \in (0, 1]$
		\STATE Set $\zz_{k+1} = (1 - \gamma_k) \yy_{k} + \gamma_k \xx_{k}$
		\STATE Compute $\xx_{k+1} = \text{InexactProx}(\xx_{k}, \zz_{k+1}, \beta_k, \eta_k)$
		for some $\beta_k \geq 0$ and $\eta_k \geq 0$
		\STATE Set $\yy_{k+1} = (1 - \gamma_k) \yy_{k} + \gamma_k \xx_{k+1}$
		\ENDFOR
	\end{algorithmic}
	\caption{Accelerated Method}
	\label{alg:fw-fully-comp}
\end{algorithm}
Note that condition~\eqref{InexactProxGuarantee} implies $P(\uu^+) \leq P(\uu) + \eta, \; \forall \uu \in \X$.
Formally, the main convergence result characterizing Algorithm~\ref{alg:fw-fully-comp} is the following.
\begin{restatable}{thm-rest}{ThmAcceleration}
	\label{TheoremAcceleration}
	Let Assumptions~\ref{AssumptionF}, ~\ref{AssumptionLipCont}, and~\ref{AssumptionConvexity} be satisfied. We choose $\gamma_k\coloneqq\frac{3}{k+3}$, $\beta_k\coloneqq cF(\mat{L})\gamma_k$ and $\eta_k \coloneqq \frac{\delta}{3(k+1)(k+2)}$ where $\delta > 0$ and $c \geq 0$ are chosen constants, and $F(\mat{L}) \coloneqq \sup_{\xx \in \X} F(\mat{L}, \xx)$. Then, for all $k \geq 1$ it holds that
	\begin{equation*}
		\ba{rcl}
		\varphi(\yy_k) - \varphi^{\star} & \leq & \frac{\delta + 8cF(\mat{L})\diam^2}{(k+2)(k+3)} +  \frac{2\max\{0, 1-c\}F(\mat{L})\diam^2 }{k+3}.
		\ea
	\end{equation*}    
\end{restatable}
The proof of Theorem~\ref{TheoremAcceleration} comes from a natural sequence of steps involving the properties of the operators and the approximate optimality of $\xx_{k+1}$. The crucial step in attaining the improved convergence is the choice of parameters $\gamma_k$, $\beta_k$ and $\eta_k$. Notably, the decay speed required of $\eta_k$ is quadratic, meaning that the subproblems are solved with fast-increasing accuracy and at the cost of additional time spent in the subsolver. The constant $\delta$ allows us to fine-tune the accuracy required for the first several iterations of the algorithm, where we can demand a lower accuracy. In practice, we can always choose $\delta = 1$ as a universal rule,
and the optimal choice is $\delta = F(\mat{L}) \diam^2$ when these parameters are known.
The factor $c F(\mat{L})$ in the definition of $\beta_k$ can be interpreted as the quality of the approximation of the Lipschitz constant for our problem.  Namely, it is exactly computed for $c = 1$, and over or underestimated for $c > 1$ and $c \in (0,1)$ respectively. 

We describe each of the bounding terms independently: the first is highly reminiscent of the usual bounds accompanying FW-type algorithms in terms of constants, albeit now with quadratic decay speed. The second term indicates the behavior of the algorithm as a function of $c$: overestimation of $F(\mat{L})$ ensures quadratic rates of convergence, since the second term becomes negative.  Conversely, underestimation of $F(\mat{L})$ brings us back into the familiar FW convergence regime of $\bigO{1/k}$ as the second term becomes positive. The extreme case $c=0$ (and hence $\beta_k = 0$) essentially reduces Algorithm~\ref{alg:fw-fully-comp} to Algorithm~\ref{alg:basic}, since the projection subproblem reduces to problem~\eqref{MainSubproblem} which we assume to be easily solvable. We therefore have robustness in terms of choosing  the parameter $c$ and the exact knowledge of $F(\mat{L})$ is not needed, even though it may come at the cost of a slower convergence. In contrast, classical Fast Gradient Methods are usually very sensitive to such parameter choices~\citep{devolder2013exactness}.

Theorem~\ref{TheoremAcceleration} provides an accelerated rate on the iterates $\yy_k$ -- an analogous result to that of~\citet{lan2016conditional} albeit under a different oracle. This convergence rate is conditioned on the subsolver returning an $\eta_k$-inexact solution to the projection subproblem and therefore any subsolver satisfying the condition can achieve this rate. As with any optimization algorithm, convergence guarantees may also be stated in terms of the oracle complexity required to reach $\epsilon$ accuracy. For Algorithm~\ref{alg:fw-fully-comp} all the oracle calls are deferred to the subsolver $\text{InexactProx}$, which we describe and analyze in the next section.

\section{Solving the Proximal Subproblem}
\label{SectionInexactProx}
We now provide an instance of the $\text{InexactProx}$ subsolver which fully determines the oracle complexity of the Accelerated Method (Algorithm~\ref{alg:fw-fully-comp}). It relies on a specific adaptation of Algorithm~\ref{alg:basic} to the structure of~\eqref{ProxPoint}. The quadratic regularizer is linearized and oracles of type~\eqref{MainSubproblem} are called once per inner iteration, while the Jacobian $\nabla \ff(\zz_k)$ is computed once per subsolver call. The main challenge here is to find a readily available quantity defining the exit condition of the subsolver, which we denote by $\Delta_t$.
\begin{algorithm}[ht!]
	\setstretch{1.3}
	\begin{algorithmic}
		\STATE \textbf{Initialization:} $\uu_0 = \xx$.
		\FOR{$t = 0, 1, \dots $}
		\STATE Compute $\displaystyle \begin{aligned}[t]
			\vv_{t + 1} & \in & \Argmin\limits_{\vv \in \X}
			\Bigl\{ 
			F\bigl(  \ff(\zz) + \nabla \ff(\zz) (\vv - \zz), \, \vv \bigr)
			\, + \, \beta \la \uu_t - \xx, \vv  \ra 
			\Bigr\}\\[2mm]
		\end{aligned}$
		\STATE Compute $ \displaystyle\begin{aligned}[t]
			\displaystyle \Delta_{t} \;\; &= \;\; 
			F\bigl(  \ff(\zz) + \nabla \ff(\zz) (\uu_t - \zz), \, \uu_t \bigr)  \; - \;
			F\bigl(  \ff(\zz) + \nabla \ff(\zz) (\vv_{t + 1} - \zz), \, \vv_{t + 1} \bigr) \; \\[-1mm]
			&\hspace{75mm}+ \; \beta \la \uu_t  - \xx, \uu_t - \vv_{t + 1} \ra\end{aligned} \) 
		\STATE {\textbf{if } $\Delta_t \leq \eta$ \textbf{ then return } $\uu_t$}
		\STATE Set $\alpha_t = \min\Bigl\{ 1, \, \frac{\Delta_t}{\beta \| \vv_{t + 1} - \uu_t \|_2^2} \Bigr\}\;\;\;$
		and $\;\; \;\uu_{t + 1} = \alpha_t \vv_{t + 1} + (1 - \alpha_t) \uu_t$
		\ENDFOR
	\end{algorithmic}
	\caption{InexactProx($\xx$, $\zz$, $\beta$, $\eta$)
	}
	\label{alg:InexactProx}
\end{algorithm}

The parameters of Algorithm~\ref{alg:InexactProx} are fully specified, and the stopping condition depends on $\Delta_t \geq P(\uu_t) - P^{\star}$, which is a meaningful progress measure. The algorithm selects its stepsize via closed-form line search to improve practical performance. When $F(\uu) \equiv \uu^{(1)}$, this procedure recovers the classical FW algorithm with line search applied to problem~\eqref{ProxPoint}.

We prove two results in relation to Algorithm~\ref{alg:InexactProx}: its convergence rate and the total oracle complexity of Algorithm~\ref{alg:fw-fully-comp} when using Algorithm~\ref{alg:InexactProx} as the subsolver. The rate and analysis are similar to the ones for the Basic Method,
utilizing additionally the form of the proximal subproblem.
\begin{restatable}{thm-rest}{ThmSubsolver}
	\label{TheoremSubsolver}
	Let Assumptions~\ref{AssumptionF}, ~\ref{AssumptionLipCont}, and~\ref{AssumptionConvexity} be satisfied. 
	Then, for all $t \geq 1$ it holds that
	$$
	\ba{c}
	P(\uu_t) - P^{\star} \;\; \leq  \;\; \frac{2\beta\diam^2}{t+1}  \qquad \text{ and } \qquad 
	\min\limits_{1 \leq i \leq t} \Delta_t \;\; \leq  \;\; \frac{6 \beta \diam^2 }{t}.
	\ea
	$$
	Consequently, Algorithm~\ref{alg:InexactProx} returns an $\eta$-approximate solution according to condition~\eqref{InexactProxGuarantee} after at most $\cO\Bigl( \frac{\beta \diam^2}{\eta}  \Bigr)$	iterations.
\end{restatable}

We note that oracle~\eqref{MainSubproblem} is called once per inner iteration and the Jacobian $\nabla \ff(\zz_k)$ is computed once per subsolver call. 
In particular, when using Algorithm~\ref{alg:InexactProx} as a subsolver, our Accelerated Method achieves the optimal number of $\bigO{\epsilon^{-1/2}}$ Jacobian computations typical of smooth and convex optimization, while maintaining a $\bigO{\epsilon^{-1}}$ complexity for the number of calls to oracle~\eqref{MainSubproblem}. The results are stated in the following corollary.
\begin{restatable}{cor-rest}{CorSubsolver}
	\label{CorSubsolver}
Consider the optimal choice of parameters
	for Algorithm~\ref{alg:fw-fully-comp}, that is $c \coloneqq 1$ and $\delta \coloneqq F(\mat{L}) \diam^2$.
	Then, solving problem~\eqref{MainProblem} with $\varepsilon$ accuracy $\varphi(\yy_k) - \varphi^{\star} \leq \varepsilon$, requires
	$ 
	\cO\bigl(  \sqrt{ \frac{F(\mat{L}) \diam^2}{ \varepsilon} } \bigr)
	$
	computations of $\nabla \ff$. In addition, the total number of calls to oracle~\eqref{MainSubproblem} is
	$
	\cO\bigl(  \frac{F(\mat{L}) \diam^2}{ \varepsilon}\bigr).
	$
\end{restatable}

Finally, we note that for smaller values of parameter $c \in [0, 1]$
in Algorithm~\ref{alg:fw-fully-comp}
(underestimating the Lipschitz constant),
the complexity of InexactProx procedure improves.
Thus, for $c = 0$ we have $\beta = 0$ (no regularization)
and Algorithm~\ref{alg:InexactProx} finishes after just \textit{one step}.

\section{Experiments}
\label{SectionExperiments}
\begin{figure*}[t!]
	\centering
	\begin{minipage}[t]{\textwidth}
		\setlength{\lineskip}{0pt}
		\subfigure[]{
			\centering
			\includegraphics[width=.44\linewidth]{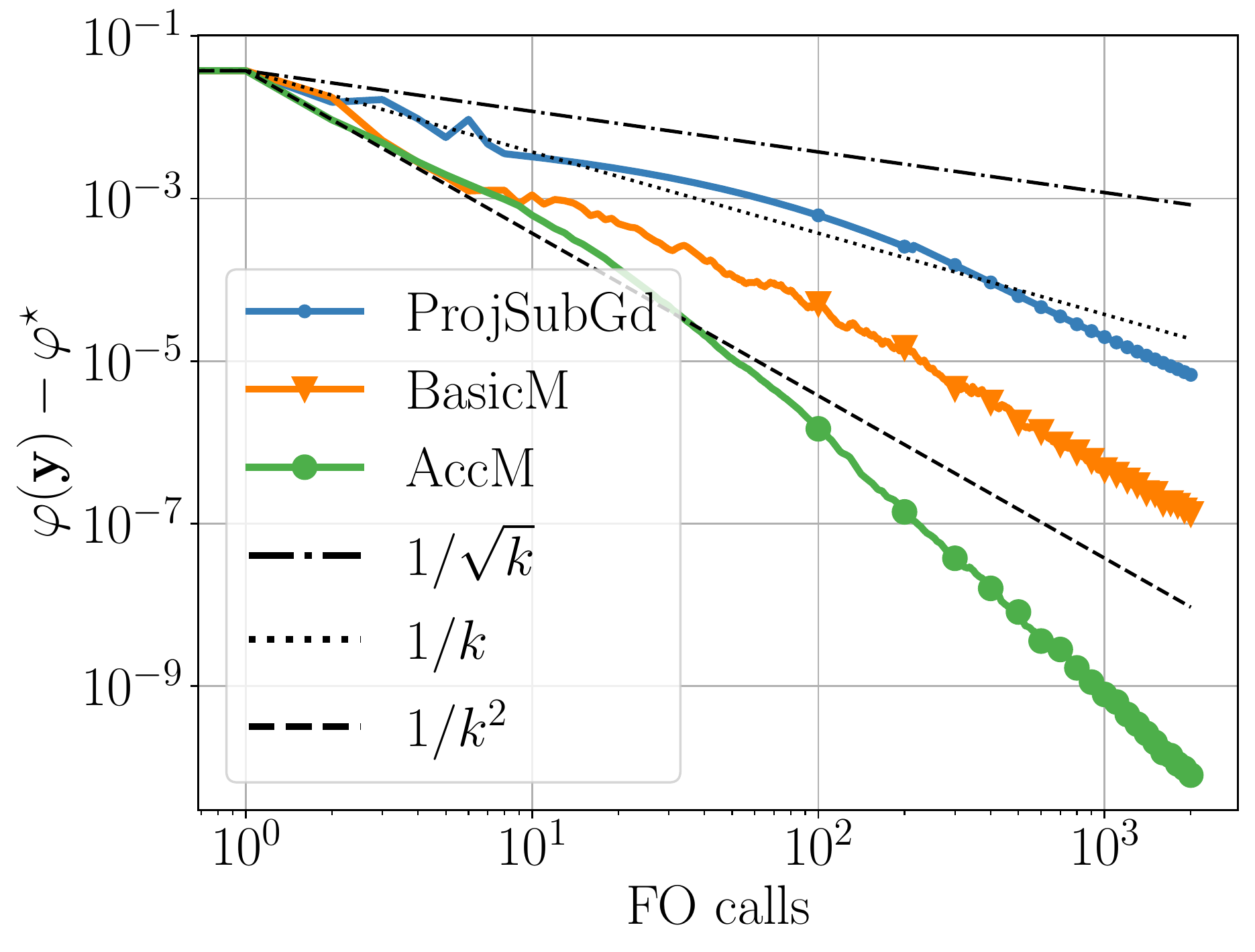}
			\label{FigSimplexFO}
		}\hspace{5mm}
		\subfigure[]{
			\centering
			\includegraphics[width=.44\linewidth]{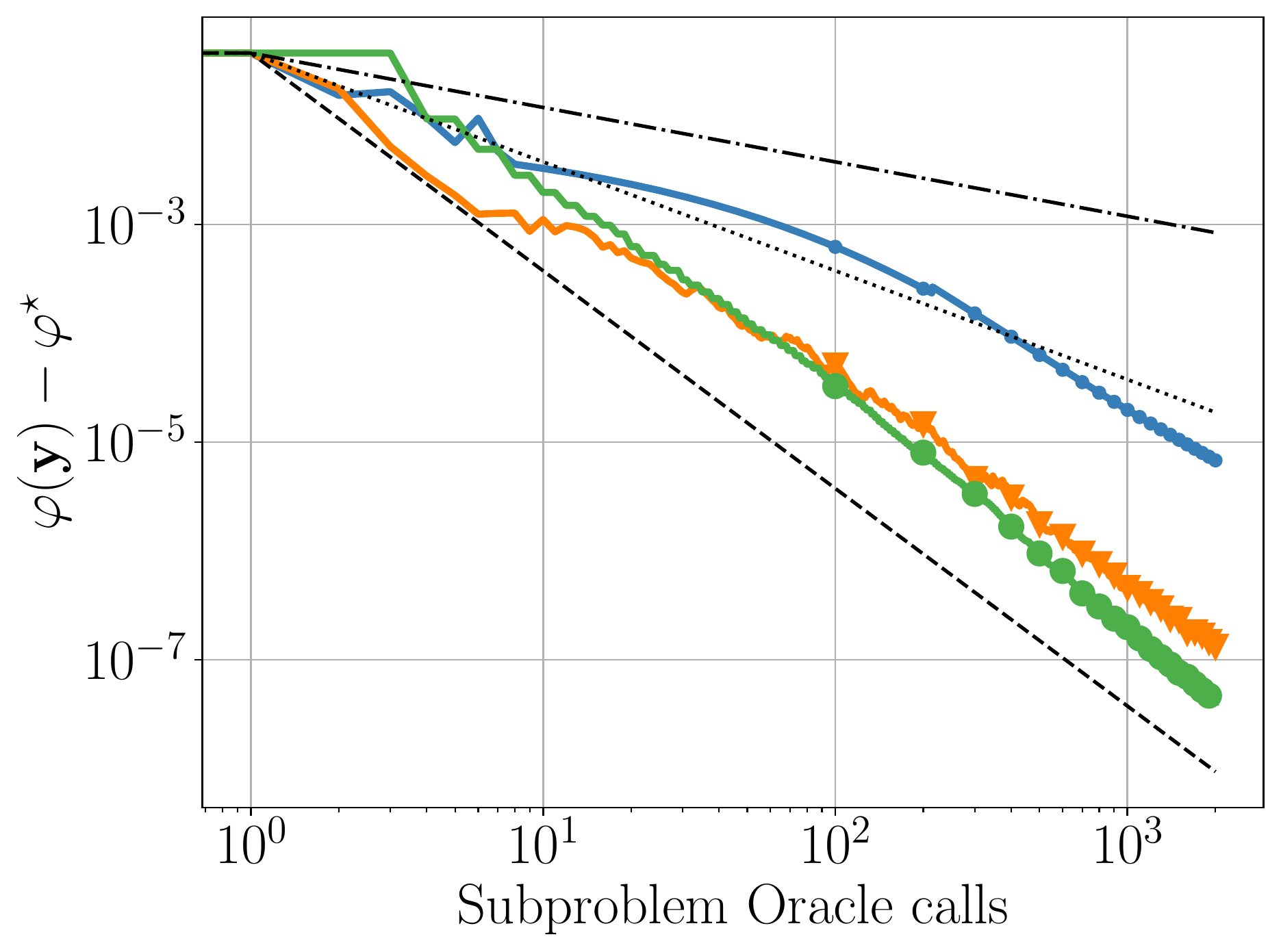}
			\label{FigSimplexLMO}
		}
		\vspace{-3mm}
		\caption{Convergence of the Basic and Accelerated methods against the Projected Subgradient baseline for problem~\eqref{exp-simplex-problem}, along with relevant theoretical rates.} 
		\vspace{-2mm}
	\end{minipage}
\end{figure*}
The experiments are implemented in \texttt{Python 3.9} and run on a MacBook Pro M1 with 16 GB RAM. For both experiments we use the Projected Subgradient Method as a baseline~\citep{shor1985minimization}, with a stepsize of $\frac{p}{\sqrt{k}}$ where $p$ is tuned for each experiment. The  \texttt{CVXPY} library~\citep{diamond2016cvxpy} is used to solve subproblems of type~\eqref{MainSubproblem}. The random seed for our experiments is always set to $\texttt{666013}$, and we set $c=1$ since we can analytically compute the Lipschitz constants or their upper bounds.

\subsection{Minimization Over the Simplex} We consider the following optimization problem
\begin{align}
	\label{exp-simplex-problem}
	\min_{\xx \in \X} \left\{ \max_{i = 1, n} \xx^\top \AA_i \xx  - \bb_i^\top\xx \right\}, \text{ for } \X = \Delta_d, \subseteq \R^d,
\end{align}
where $ \AA_i \in \R^{d \times d}$ are random PSD matrices and $\bb \in \R^d$. The problem conforms to Example~\ref{ExampleMax}, and we use $d = 500$ and $n = 10$. We generate $\AA_i = \QQ_i \DD \QQ_i^{\top}$, where $\DD$ is a diagonal matrix of eigenvalues decaying linearly from $1$ to $10^{-6}$, and $\QQ_i$ is a randomly generated orthogonal matrix using the \texttt{scipy.stats.ortho\_group} method~\citep{mezzadri2006generate}. The vectors $\bb_i$, which determine the position of the quadratics in space, are set as follows: $\bb_i = \ee_i \cdot 10$, $\bb_9 = \mathbf{0}$ (the origin), $\bb_{10} = \mathbf{1} \cdot 10$ (the all ones vector multiplied by $10$). We set $\delta = 0.2$ in the Accelerated Method (see Theorem~\ref{TheoremAcceleration}) and settle for  $p=1.42$ following tuning of the Subgradient Method. Finally, we set $\xx_0=\ee_3 \in \Delta_d$ for all methods. 

The convergence results in terms of FO oracles and oracles of type~\eqref{MainSubproblem} are shown in Figure~\ref{FigSimplexFO}~and~\ref{FigSimplexLMO}, respectively. The figures highlight the improvement in terms of the number of FO calls, while showing comparable performance in terms of subproblem oracle calls, as predicted by our theory. 

\begin{figure*}[t!]
	\centering
	\begin{minipage}[t]{\textwidth}
		\setlength{\lineskip}{0pt}
		\subfigure[]{
			\includegraphics[width=.45\linewidth]{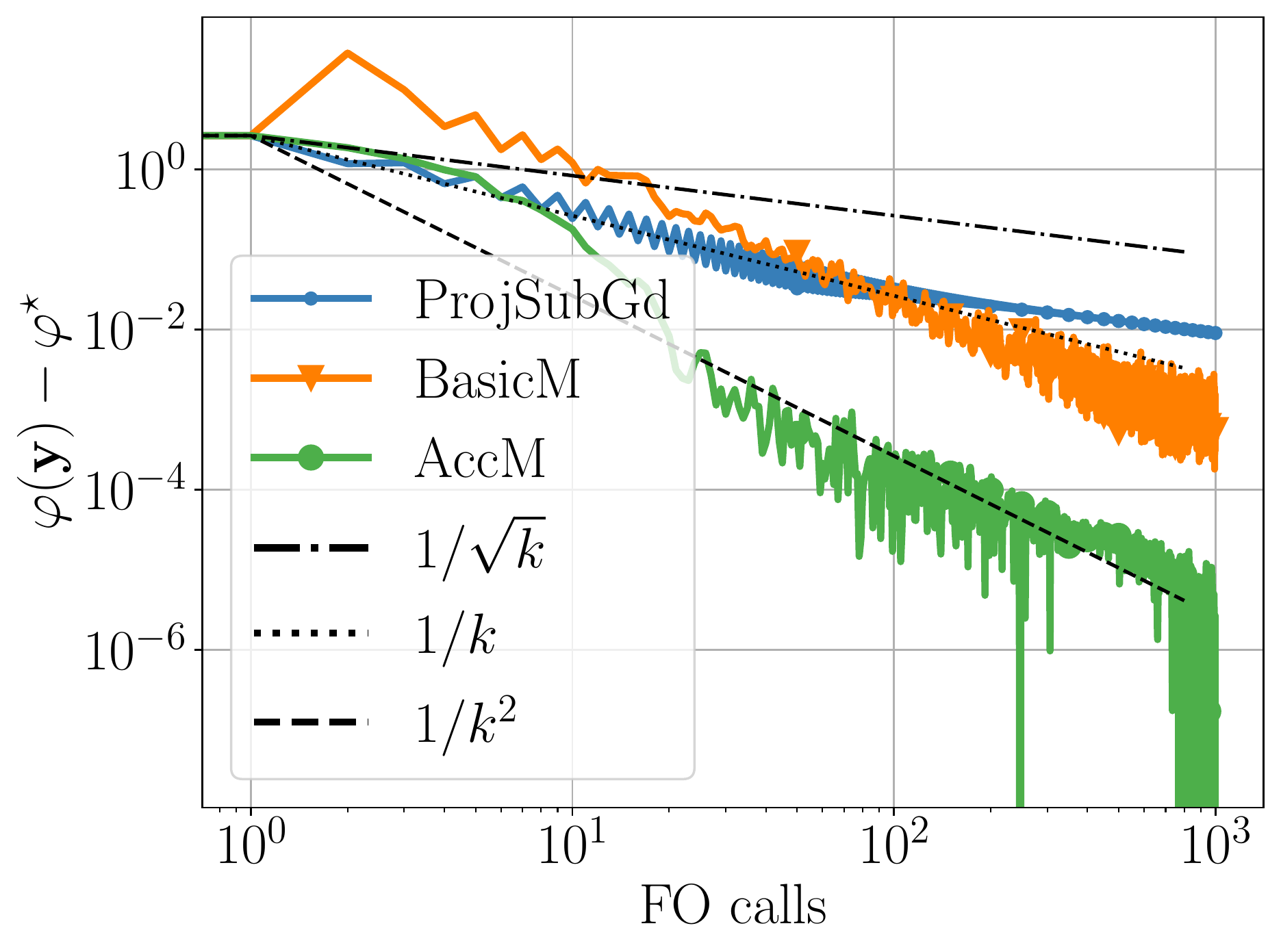}
			\label{FigNucFO}
		}
		\subfigure[]{
			\includegraphics[width=.45\linewidth]{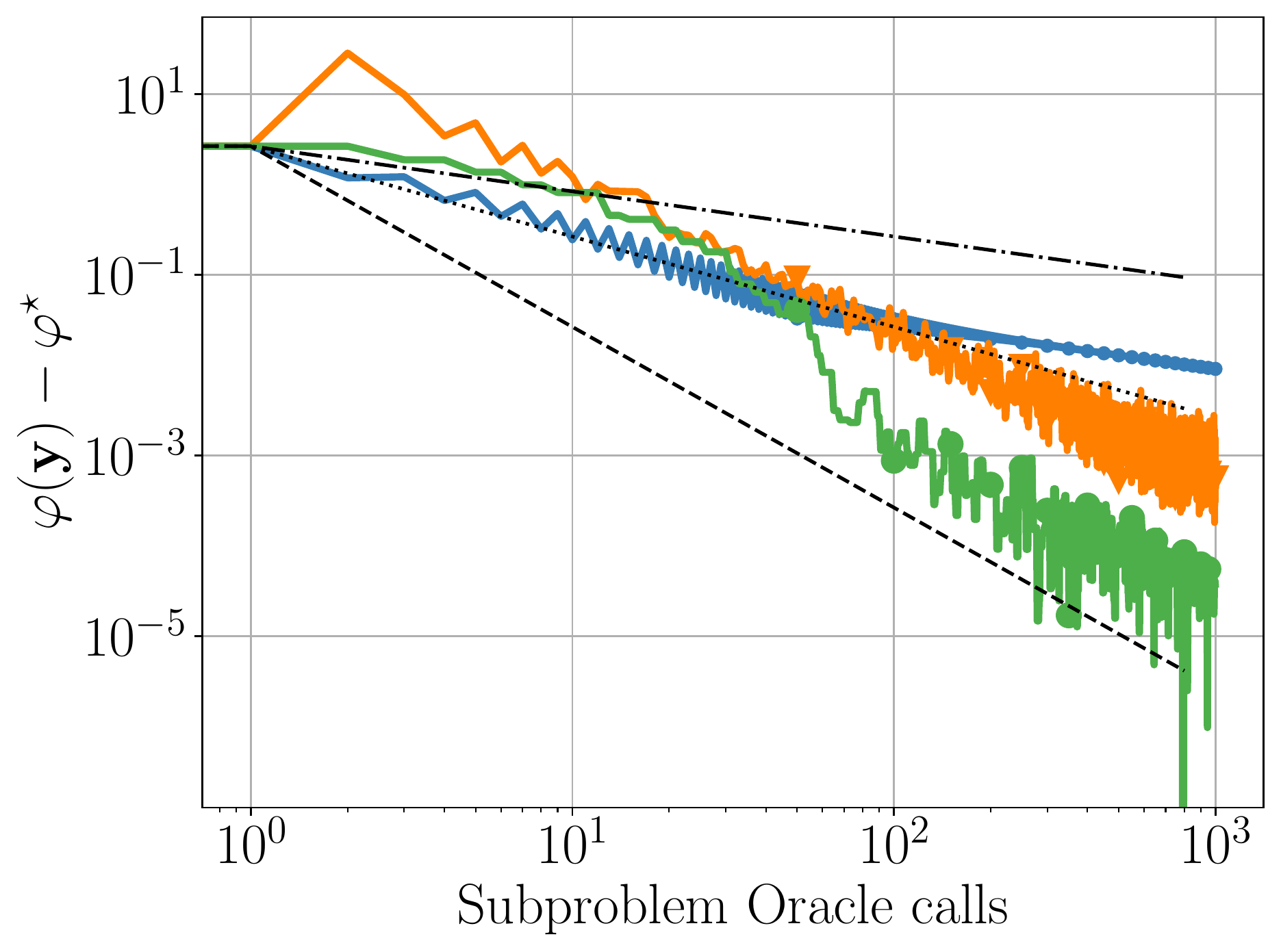}
			\label{FigNucLMO}
		}
		\caption{Convergence of the Basic and Accelerated methods against the Projected Subgradient baseline for problem~\eqref{probnucnorm}, along with relevant theoretical rates.} 
	\end{minipage}
\end{figure*}
\subsection{Minimization Over the Nuclear Norm Ball} We consider the following optimization problem
\begin{align}
	\label{probnucnorm}
	\min_{\XX \in \X}\left\{ \max_{i = 1, n} \sum_{(k, l) \in \Omega_i} \left(\XX_{k,l} - \AA^{(i)}_{k,l}\right)^2 \right\}, \text{ for } \X := \{\XX \in \R^{d \times m}, \; \norm{\XX}_* \leq r \} 
\end{align}
Formulation~\eqref{probnucnorm} models a matrix completion scenario where we wish to recover an $\XX^{\star}$ that minimizes the largest error within a given \emph{set} of matrices $\AA^{(i)}$. The matrices $\AA^{(i)}$ are only partially revealed through a set of corresponding indices $\Omega_i$. This problem conforms to Example~\ref{ExampleMax} and we use $d =30 $, $m=10$, $r=7$, where $r$ is the rank of matrices $\AA^{(i)}$. The data is generated in identical fashion as in Section~5.2~of~\citep{lan2016conditional} on Matrix Completion. We set $\delta = 100$ in the Accelerated Method (see Theorem~\ref{TheoremAcceleration}) and settle for  $p=0.2$ following tuning of the Subgradient Method. Finally, we set $\xx_0=\mathbf{0}^{d \times m} \in \X$ for all methods. 

The convergence results in terms of FO oracles and oracles of type~\eqref{MainSubproblem} are shown in Figure~\ref{FigNucFO}~and~\ref{FigNucLMO}, respectively. The figures highlight the improvement in terms of the number of FO calls, while showing comparable performance in terms of subproblem oracle calls, as predicted by our theory.

\section{Conclusion} Our work illustrates how improved convergence rates may be attained by assuming precise structure within a class of objectives (e.g., non-differentiable ones). Moreover, it shows how a simple principle such as linearizing the differentiable components of a function composition can be used to create more benign subproblems that are efficiently solved. Interesting future work may address relaxing the assumptions on $F$, extending this framework to stochastic settings, and meaningfully interpreting the quantity $\Delta_k$ in the non-convex setting.

\acks{The authors thank Aditya Varre and the anonymous reviewers for their helpful feedback on the manuscript. This work was partly supported by the Swiss State Secretariat for Education, Research and Innovation (SERI) under contract number 22.00133.}

\bibliography{references}

\newpage
\appendix
\section{Proofs}
\label{SectionProofs}
Assumption~\ref{AssumptionSmooth} implies global progress bounds on our fully composite objective with an inner linearization of $\ff$, as stated in the following Lemma~\ref{LemGlobalBounds}. This lemma provides a basis for all our convergence results. 
\begin{restatable}{lem-rest}{LemGlobalBounds}
	\label{LemGlobalBounds}
	Let $\xx, \yy \in \X$ and $\gamma \in [0, 1]$.
	Denote $\yy_{\gamma} = \xx + \gamma(\yy - \xx)$.
	Then, it holds
	\beq \label{PhiContrBound}
	\ba{rcl}
	\varphi(\yy_{\gamma}) & \leq & 
	F\bigl( \ff(x) + \nabla \ff(\xx)(\yy_{\gamma} - \xx), \, \yy_{\gamma}  \bigr)
	\; + \;
	\frac{\gamma^2}{2} \mathcal{S}.
	\ea
	\eeq
\end{restatable}
\begin{proof}
Note that the subhomogenity assumption \eqref{FSubH}
is equivalent to the following useful inequality for the outer component of the objective
see (Theorem 7.1 in \cite{doikov2022high}):
\beq \label{FSubAdd}
\ba{rcl}
F(\uu + t \vv, \xx) & \leq & F(\uu, \xx) + t F(\vv, \xx), 
\qquad
\forall \uu, \vv \in \R^n, \; \xx \in X, \; t \geq 0.
\ea
\eeq
Then, we have
$$ 
\ba{cl}
& \varphi( \yy_{\gamma} ) 
\;\; \equiv \;\;
F( \ff( \yy_{\gamma}), \, \yy_{\gamma}  )  \\[5pt]
& \; = \;
F\bigl( \ff(\xx) + \nabla \ff(\xx)(\yy_{\gamma} - \xx)
+ \frac{\gamma^2}{2} 
\cdot \frac{2}{\gamma^2}
\bigl[ \ff(\yy_{\gamma}) - \ff(\xx) - \nabla f(\xx)(\yy_{\gamma} - \xx) \bigr]
, \, \yy_{\gamma} \bigr) \\[5pt]
& \overset{\eqref{FSubAdd}}{\leq} 
F\bigl( \ff(\xx) + \nabla \ff(\xx)(\yy_{\gamma} - \xx), \, \yy_{\gamma} \bigr) 
+ \frac{\gamma^2}{2} 
F\bigl( \frac{2}{\gamma^2}
\bigl[ \ff(\yy_{\gamma}) - \ff(\xx) - \nabla f(\xx)(\yy_{\gamma} - \xx) \bigr] , \, \yy_{\gamma} \bigr)  \\[10pt]
& \; \leq \;
F\bigl( \ff(\xx) + \nabla \ff(\xx)(\yy_{\gamma} - \xx), \, \yy_{\gamma} \bigr) 
+ \frac{\gamma^2}{2} \mathcal{S},
\ea
$$
which is the desired bound.
\end{proof}

\subsection{Proof of Theorem \ref{TheoremBasicConvex}}

\ThmBasicConvex*
\begin{proof}
	Indeed, for one iteration of the method, we have
	$$
	\ba{rcl}
	\varphi(\yy_{k + 1}) & \overset{\eqref{PhiContrBound}}{\leq} &
	F\bigl( \ff(\yy_k) + \nabla \ff(\yy_k)(\yy_{k + 1} - \yy_k), \, \yy_{k + 1}  \bigr)
	+ \frac{\gamma_k^2}{2} \mathcal{S} \\
	\\
	& = &
	F\bigl( (1 - \gamma_k) \ff(\yy_k)
	 + \gamma_k( \ff(\yy_k) + \nabla \ff(\yy_k)(\xx_{k + 1} - \yy_k) ), \\
	 \\
	 & &
	 \quad\;
	 (1 - \gamma_k) \yy_k + \gamma_k \xx_{k + 1}
	 \bigr) + \frac{\gamma_k^2}{2} \mathcal{S} \\
	 \\
	 & \overset{(*)}{\leq} &
	 \varphi(\yy_k) + \gamma_k\Bigl[ 
	 F\bigl(  \ff(\yy_k) + \nabla \ff(\yy_k)(\xx_{k + 1} - \yy_k), \xx_{k + 1} \bigr)
	 - \varphi(\yy_k)
	 \Bigr]
	+ \frac{\gamma_k^2}{2} \mathcal{S}  \\
	\\
	& \equiv &
	\varphi(\yy_k) - \gamma_k \Delta_k + \frac{\gamma_k^2}{2} \mathcal{S},
	\ea
	$$
	where we used in $(*)$ that $F(\cdot, \cdot)$ is jointly convex.
	Hence, we obtain the following inequality for the progress of one step, for $k \geq 0$:
	\beq \label{GapKey}
	\ba{rcl}
	\varphi(\yy_k) - \varphi(\yy_{k + 1})
	& \geq & 
	\gamma_k \Delta_k
	- \frac{\gamma_k^2}{2} \mathcal{S}.
	\ea
	\eeq	
	
	Now, let us choose use an auxiliary sequence $A_k := k \cdot (k + 1)$ and
	$a_{k + 1} := A_{k + 1} - A_k = 2(k + 1)$.
	Then,
	$$
	\ba{rcl}
	\frac{a_{k + 1}}{A_{k + 1}} & = & \frac{2}{2 + k},
	\ea
	$$
	which is one of the possible choices for $\gamma_k$.
	Note that for the other choice, we set $\gamma_k = \min\bigl\{1, \frac{\Delta_k}{\mathcal{S}} \bigr\}$,
	which maximizes the right hand side of \eqref{GapKey}
	with respect to $\gamma_k \in [0, 1]$.
	Hence, in both cases we have that
	\beq \label{Th1GapBound}
	\ba{rcl}
	\varphi(\yy_k) - \varphi(\yy_{k + 1})
	& \geq & 
	\frac{a_{k + 1}}{A_{k + 1}} \Delta_k
	- \frac{a_{k + 1}^2}{2 A_{k + 1}^2} \mathcal{S},
	\ea
	\eeq
	or, rearranging the terms,
	$$
	\ba{rcl}
	A_{k + 1} \bigl[ \varphi(\yy_{k + 1}) - \varphi^{\star} \bigr]
	& \overset{\eqref{Th1GapBound}}{\leq} &
	A_{k + 1} \bigl[ \varphi(\yy_k) - \varphi^{\star}]
	- a_{k + 1} \Delta_k
	+ \frac{a_{k + 1}^2}{2A_{k + 1}} \mathcal{S} \\
	\\
	& \overset{\eqref{DeltaKLower}}{\leq} &
	A_{k} \bigl[ \varphi(\yy_k) - \varphi^{\star}]
	+ \frac{a_{k + 1}^2}{2A_{k + 1}} \mathcal{S}.
	\ea
	$$
	Telescoping this bound for the first $k \geq 1$ iterations, we get
	\beq \label{Th1FuncRes}
	\ba{rcl}
	\varphi(\yy_k) - \varphi^{\star} & \leq &
	\frac{\mathcal{S}}{2 A_k} \cdot \sum\limits_{i = 1}^k \frac{a_i^2}{A_i}
	\;\; = \;\;
	\frac{2 \mathcal{S}}{k (k + 1)} \cdot \sum\limits_{i = 1}^k \frac{i}{i + 1}
	\;\;
	\leq \;\;
	\frac{2 \mathcal{S}}{k + 1}.
	\ea
	\eeq
	
	It remains to prove the convergence in terms of the accuracy measure $\Delta_k$.
	For that, we telescope the bound \eqref{Th1GapBound}, which is
	\beq \label{Th1GapB1}
	\ba{rcl}
	a_{k + 1} \Delta_k & \leq & a_{k + 1} \varphi(\yy_k)
	+ A_k \varphi(\yy_k) - A_{k + 1} \varphi(\yy_{k + 1})
	+ \frac{a_{k + 1}^2}{A_{k + 1}} \frac{\mathcal{S}}{2},
	\ea
	\eeq
	for the $k \geq 1$ iterations, and use the convergence for the functional residual \eqref{Th1FuncRes}:
	$$
	\ba{rcl}
	\sum\limits_{i = 1}^k a_{i + 1}
	\cdot \min\limits_{1 \leq i \leq k} \Delta_i
	& \leq & 
	\sum\limits_{i = 1}^k a_{i + 1} \Delta_i \\
	\\
	& \overset{\eqref{Th1GapB1}}{\leq} &
	a_1 \bigl[ \varphi(\yy_1) - \varphi^{\star}  \bigr]
	+ \sum\limits_{i = 1}^k
	a_{i + 1} \bigl[  \varphi(\yy_i) - \varphi^{\star} \bigr]
	+ \frac{\mathcal{S}}{2} \sum\limits_{i = 1}^k \frac{a_{i + 1}^2}{A_{i + 1}} \\
	\\
	& \overset{\eqref{Th1FuncRes}}{\leq} &
	2 \mathcal{S} \cdot 
	\Bigl( 
	1 + \sum\limits_{i = 1}^k \frac{a_{i + 1}}{i + 1}
	+ \sum\limits_{i = 1}^k \frac{i}{i + 1} 
	\Bigr)
	\;\; \leq \;\;
	2\mathcal{S} \cdot (1 + 3k).
  	\ea
	$$ 
	To finish the proof, we need to divide both sides by
	$\sum\limits_{i = 1}^k a_{i + 1} = A_{k + 1} - a_1 = k (3 + k)$.
\end{proof}

\subsection{Proof of Theorem \ref{TheoremGap}}

\ThmGap*
\begin{proof}
	As in the proof of the previous theorem, our main 
	inequality \eqref{GapKey} on the progress of one step is:
	$$
	\ba{rcl}	
	\varphi(\yy_k) - \varphi(\yy_{k + 1})
	& \geq & 
	\gamma_k \Delta_k
	- \frac{\gamma_k^2}{2} \mathcal{S},
	\ea
	$$
	where we can substitute $\gamma_k = \frac{1}{\sqrt{k + 1}}$
	for the both strategies of choosing this parameter.
	
	Summing up this bound for the first $k + 1$ iterations, we obtain
	\beq \label{DeltaBound}
	\ba{rcl}
	\sum\limits_{i = 0}^k \gamma_i \Delta_i
	& \leq& 
	\varphi(\yy_0) - \varphi(\yy_{k + 1}) 
	+ \frac{\mathcal{S}}{2} \sum\limits_{i = 0}^k \gamma_i^2.
	\ea
	\eeq
	Using the bound $\varphi(\yy_{k + 1}) \geq \varphi^{\star}$ and our value of $\gamma_i$, we get
	$$
	\ba{rcl}
 \min\limits_{0 \leq i \leq k} \Delta_i \cdot \sqrt{k + 1}
 & \leq &
 	\sum\limits_{i = 0}^k \frac{\Delta_i}{\sqrt{1 + i}}
 	\;\; \overset{\eqref{DeltaBound}}{\leq} \;\;
 	\varphi(\yy_0) - \varphi^{\star} + \frac{\mathcal{S}}{2} \sum\limits_{i = 0}^k \frac{1}{1 + i} \\
 	\\
 	& \leq & 
 	\varphi(\yy_0) - \varphi^{\star}  + \frac{\mathcal{S}}{2} (1 + \ln(k + 1)),
	\ea
	$$
	which is \eqref{ConvergenceGapGeneral}.
\end{proof}

\subsection{Proof of Theorem \ref{TheoremAcceleration}}

\ThmAcceleration*
\begin{proof}
	Let us consider one iteration of the method, for some $k \geq 0$.
	
	Since all the components of $\ff$ have the Lipschitz continuous gradients,
	it hold that
	$$
	\ba{rcl}
	\ff(\yy_{k + 1}) & \leq &
	\ff(\zz_{k + 1}) + \nabla \ff(\zz_{k + 1})(\yy_{k + 1} - \zz_{k + 1})
	+ \frac{\mat{L}}{2} \|\yy_{k + 1} - \zz_{k + 1}\|^2,
	\ea
	$$
	where the vector inequality is component-wise.
	Then, using the properties of $F$, we have
	$$
	\ba{rcl}
	\varphi(\yy_{k + 1}) & = & F(\ff(\yy_{k + 1}), \yy_{k + 1}) \\
	\\
	& \overset{\eqref{FMonotone}, \eqref{FSubAdd}}{\leq} &
	F( \ff(\zz_{k + 1}) + \nabla \ff(\zz_{k + 1}) (\yy_{k + 1} - \zz_{k + 1}), \yy_{k + 1}   )
	+ \frac{F(\mat{L})}{2}\| \yy_{k + 1} - \zz_{k + 1}\|^2 \\
	\\
	& = & 
	F\bigl( (1 - \gamma_k) \bigl[ \ff(\zz_{k + 1}) +   \nabla \ff(\zz_{k + 1})(\yy_k - \zz_{k + 1}) \bigr] \\
	\\
	& & \,\; \qquad
	\; + \; \gamma_k \bigl[ \ff(\zz_{k + 1}) + \nabla \ff(\zz_{k + 1}) (\xx_{k + 1} - \zz_{k + 1}) \bigr],  \\
	\\
	& & \qquad \qquad
	(1 - \gamma_k) \yy_k + \gamma_k \xx_{k + 1} \bigr)
	\; + \; \frac{\gamma_k^2 F(\mat{L})}{2}\| \xx_{k + 1} - \xx_k\|^2 \\
	\\
	& \leq & 
	(1 - \gamma_k) F( \ff(\zz_{k + 1}) + \nabla \ff( \zz_{k + 1})( \yy_k - \zz_{k + 1}), \, \yy_k  ) \\
	\\
	& & + \;
	\gamma_k F( \ff(\zz_{k + 1}) +  \nabla \ff(\zz_{k + 1})( \xx_{k + 1} - \zz_{k + 1} ), \, \xx_{k + 1}   )
	\; + \; \frac{\gamma_k^2 F(\mat{L})}{2}\| \xx_{k + 1} - \xx_k\|^2 \\
	\\
	& \leq &
	(1 - \gamma_k) \varphi(\yy_{k})
	\; + \; \gamma_k F( \ff(\zz_{k + 1}) +  \nabla \ff(\zz_{k + 1})( \xx_{k + 1} - \zz_{k + 1} ), \, \xx_{k + 1}   )\\
	\\
	& & \, \; \qquad
	\; + \; \frac{\gamma_k^2 F(\mat{L})}{2}\| \xx_{k + 1} - \xx_k\|^2,
	\ea
	$$
	where the second equality comes from the update rule of $\yy_{k+1}$, the second inequality comes from the joint convexity in Assumption~\ref{AssumptionF}, the third inequality comes from convexity of the components of $\ff$ and monotonicity of $F$.
	
	Since we are introducing a norm-regularized minimization subproblem for the purpose of acceleration, the term $\frac{\gamma_k^2 F(\mat{L})}{2}\| \xx_{k + 1} - \xx_k\|^2$ can be further upper bounded using $\eta_k$-approximate guarantee \eqref{InexactProxGuarantee},
	as follows, for any $\xx \in \X$:
	$$
	\ba{rcl}
	\frac{\gamma_k^2 F(\mat{L})}{2}\| \xx_{k + 1} - \xx_k\|^2
	& = &
	\left(\frac{\gamma_k^2F(\mat{L})}{2} - \frac{\beta_k\gamma_k}{2} \right)  
	\| \xx_{k+1} - \xx_k  \|^2  +   \frac{\beta_k\gamma_k}{2} \| \xx_{k+1} - \xx_k \|^2 \\
	\\
	& \leq &
	\frac{\beta_k\gamma_k}{2} \| \xx_{k+1} - \xx_k \|_2^2  + \frac{\gamma_k^2F(\mat{L})(1 - c)}{2}\| \xx_{k+1} - \xx_k  \|^2 \\
	\\
	& = &
	\frac{\beta_k\gamma_k}{2} \left( \|{\xx - 
		\xx_{k}}\|_2^2 - \|{\xx - \xx_{k+1} }\|^2 - 2\langle \xx_{k} - \xx_{k+1} , \xx_{k+1} - \xx \rangle\right) \\[2mm]
		& &+ \frac{\gamma_k^2F(\mat{L})(1 - c)}{2}\| \xx_{k+1} - \xx_k  \|^2 \nn \\
	\\
	& \overset{\eqref{InexactProxGuarantee}}{\leq} &
	\frac{\beta_k\gamma_k}{2} \left( \|{\xx - 
		\xx_{k}}\|_2^2 - \|{\xx - \xx_{k+1} }\|^2 \right) + \frac{\gamma_k^2F(\mat{L})(1 - c)}{2}\| \xx_{k+1} - \xx_k  \|^2\\
	\\
	& &	
	+ \; 
	\gamma_k F( \ff(\zz_{k + 1}) + \nabla \ff(\zz_{k + 1})( \xx - \zz_{k + 1} ), \xx ) \\
	\\
	& &
	- \;
	\gamma_k F( \ff(\zz_{k + 1}) + \nabla \ff(\zz_{k + 1})( \xx_{k + 1} - \zz_{k + 1} ), \xx_{k + 1} ) + 
	\gamma_k \eta_k,
	\ea
	$$
	where we used our choice $\beta_k = c F(\mat{L})\gamma_k$.

	Therefore, combining these two bounds together, we obtain
	$$
	\ba{rcl}
	\varphi(\yy_{k + 1}) & \leq &
	(1 - \gamma_k) \varphi(\yy_k)
	+ \gamma_k F( \ff(\zz_{k + 1}) + \nabla \ff(\zz_{k + 1})( \xx - \zz_{k + 1} ), \xx ) \\
	\\
	& & \; + \;
	\frac{\beta_k\gamma_k}{2} \left( \|{\xx - 
		\xx_{k}}\|^2 - \|{\xx - \xx_{k+1} }\|^2 \right) + \frac{\gamma_k^2F(\mat{L})(1 - c)}{2}\| \xx_{k+1} - \xx_k  \|^2 + \gamma_k \eta_k \\
	\\
	& \leq &
	(1 - \gamma_k) \varphi(\yy_k) + \gamma_k \varphi(\xx)
	+ \frac{\beta_k\gamma_k}{2} \left( \|{\xx - 
		\xx_{k}}\|^2 - \|{\xx - \xx_{k+1} }\|^2 \right) \\[3mm]
		&& + \frac{\gamma_k^2F(\mat{L})(1 - c)}{2}\| \xx_{k+1} - \xx_k  \|^2+ \gamma_k \eta_k,
	\ea
	$$
	for all $\xx \in \X$, where we used convexity of $\ff$ and monotonicity of $F$.
	
	We now subtract $\varphi(\xx)$ from both sides, let $\xx = \xx^{\star}$ and denote $\varepsilon_k := \varphi(\yy_k) - \varphi^{\star}$,
	which gives
	\beq \label{Proofeq:recurrence-epsilon}
	\ba{rcl}
	\varepsilon_{k+1} &\leq& 
	(1 - \gamma_k) \varepsilon_{k} + \frac{\gamma_k \beta_k}{2} \left( \|{\xx_{k} - \xx^{\star} }\|^2 -  
	\|{\xx_{k+1} - \xx^{\star}}\|^2 \right) \\[3mm]
	&&+ \frac{\gamma_k^2F(\mat{L})(1 - c)}{2}\| \xx_{k+1} - \xx_k  \|^2 + \gamma_k\eta_k. 
	\ea
	\eeq
	We now move on to choosing the parameters $\gamma_k$, $\eta_k$ and $\beta_k$ in a way that allows us to accelerate. For more flexibility, we let $\gamma_k := \frac{a_{k+1}}{A_{k+1}}$, for some sequences $\{a_k\}_{k\geq 0}$ and $\{A_{k}\}_{k\geq 0}$ that will be defined later.
Then~\eqref{Proofeq:recurrence-epsilon} becomes:
	$$
	\ba{rcl}
	A_{k+1}\varepsilon_{k+1} &\leq&  
	A_0 \epsilon_0 + \sum\limits_{i = 0}^k a_{i+1}\eta_i + \frac{1}{2}\sum\limits_{i = 0}^k a_{i+1} \beta_i \left( \normsqr{\xx_{i} -\xx^{\star} } 
	-  \normsqr{\xx_{i+1} - \xx^{\star}} \right) \\[4mm]
	&&+  \frac{F(\mat{L})(1 - c)}{2} \sum\limits_{i = 0}^k \frac{a_{i+1}^2}{A_{i+1}}\| \xx_{i+1} - \xx_i  \|^2\nn \\
	\\
	& \leq & 	A_0 \epsilon_0 +
	\sum\limits_{i = 0}^k a_{i+1}\eta_i + \frac{a_1 \beta_0}{2} \normsqr{\xx_{0} -\xx^{\star} } + \frac{1}{2}\sum_{i = 1}^{k}  \left( a_{i+1} \beta_{i} - a_i \beta_{i-1} \right) \normsqr{\xx_i - \xx^{\star}}\\[5mm]
	&&+  \frac{F(\mat{L})(1 - c)}{2} \sum\limits_{i = 0}^k \frac{a_{i+1}^2}{A_{i+1}}\| \xx_{i+1} - \xx_i  \|^2\nn 
	\ea
	$$
	and therefore we have

	\beq \label{Proofeq:new-recursion}
	\ba{rcl}
	\varepsilon_{k+1} &\leq& 	 \frac{A_0 \epsilon_0}{A_{k+1}}+
	\frac{1}{A_{k+1}}\sum\limits_{i = 0}^k a_{i+1}\eta_i + \frac{a_1 \beta_0}{2A_{k+1}} \normsqr{\xx_{0} -\xx^{\star} } \\
	\\
	& &  \; + \; 
	\frac{1}{2A_{k+1}}\sum\limits_{i = 1}^{k}  \left( a_{i+1} \beta_{i} - a_i \beta_{i-1} \right) \normsqr{\xx_i - \xx^{\star} } +  \frac{F(\mat{L})(1 - c)}{2A_{k+1}} \sum\limits_{i = 0}^k \frac{a_{i+1}^2}{A_{i+1}}\| \xx_{i+1} - \xx_i  \|^2\nn.
	\ea
	\eeq
	
	We wish to choose sequences $A_k$, $a_k$, $\beta_k$ and $\eta_k$ such that we obtain a $\bigO{1/k^2}$ rate of convergence on the functional residual of $\varphi(\cdot)$. The constraint we require on the sequences is $\gamma_k F(\mat{L}) \leq \beta_k$. The following choices 
	$$
	\ba{rcl}
	\eta_k  & = & \frac{\delta}{a_{k+1}}, \; \text{ for some constant $\delta > 0$}, \\
	\\
	\beta_k & =&  c F(\mat{L}) \gamma_k,  \; \text{ for some constant $c > 0$}\\
	\\
    a_{k+1} &=& A_{k+1} - A_k  \;\;= \;\; \frac{3A_{k+1}}{k+3}, \quad A_{k+1} \;\;=\;\; (k+1)(k+2)(k+3),
	\ea
	$$
	give us the desired outcome, since equation~\eqref{Proofeq:new-recursion} becomes: 
\begin{align*}
	\varepsilon_{k+1} &\leq 
	\frac{\delta}{(k+2)(k+3)} + \frac{3cF(\mat{L})\norm{\xx_0 - \xopt}^2}{(k+1)(k+2)(k+3)}  + \frac{5ckF(\mat{L})\diam^2}{(k+1)(k+2)(k+3)}  \\[4mm]
	&\hspace{5mm} +  \frac{9F(\mat{L})(1 - c)}{2(k+1)(k+2)(k+3)} \sum\limits_{i = 0}^k (i+1)\| \xx_{i+1} - \xx_i \|^2\nn \\[4mm]
	&\leq 
	\frac{\delta}{(k+2)(k+3)} + \frac{8cF(\mat{L})\diam^2}{(k+2)(k+3)} +  \frac{2\max\{0, 1-c\}F(\mat{L})\diam^2 }{k+3}
\end{align*}
	since $a_{i+1} \beta_{i} - a_i \beta_{i-1} = \frac{9c F(\mat{L}) \left( i^2 + 5i + 4\right)}{i^2 + 5i + 6} < 9cF(\mat{L})$ and $\normsqr{\xx_i - \xx^{\star}} \leq \diam^2$.
\end{proof}

\subsection{Proof of Theorem \ref{TheoremSubsolver}}
\ThmSubsolver*
\begin{proof}
	Let us introduce our subproblem, in a general form, that is 
	\beq \label{ProofThSubMain}
	\ba{rcl}
	s^{\star} & = & 
	\min\limits_{\uu \in \X}
	\Bigl\{
	s(\uu) \; \Def \; r(\uu) + h(\uu),
	\Bigr\}
	\ea
	\eeq
	where $r(\cdot)$ is a differentiable convex function, whose
	gradient is Lipschitz continuous with constant $\beta > 0$,
	and $h(\uu)$ is a general proper closed convex function, not necessarily differentiable.
	
	In our case, for computing the inexact proximal step~\eqref{ProxPoint}, we set
	$$
	\ba{rcl}
	r(\uu) & := & \frac{\beta}{2}\|\uu - \xx\|^2, \\
	\\
	h(\uu) & := & F( \ff(\zz) + \nabla f(\zz) (\uu - \zz), \uu ),
	\ea
	$$
	for a fixed $\xx$ and $\zz$.
	
	Then, in each iteration of Algorithm~\ref{alg:InexactProx}, we compute, for $t \geq 0$:
	\beq \label{ProofThSubVt}
	\ba{rcl}
	\vv_{t + 1} & \in & \Argmin\limits_{\uu \in \X}
	\Bigl\{ 
	\;
	\la \nabla r(\uu_t), \uu \ra + h(\uu)
	\;
	\Bigr\}.
	\ea
	\eeq
	The optimality condition for this operation is (see, e.g. Theorem 3.1.23 in \cite{nesterov2018lectures})
	\beq \label{ProofThSubStat}
	\ba{rcl} 
	\la \nabla r(\uu_t), \uu - \vv_{t + 1} \ra + h(\uu) & \geq & h(\vv_{t + 1}),
	\qquad \forall \uu \in \X.
	\ea
	\eeq
	Therefore, employing the Lipschitz continuity of the gradient of $r(\cdot)$, we have
	\beq \label{ProofThSubStep1}
	\ba{rcl}
	s(\uu_{t + 1}) & \leq & 
	r(\uu_t) + \la \nabla r(\uu_t), \uu_{t + 1} - \uu_t \ra
	+ \frac{\beta}{2} \| \uu_{t + 1} - \uu_t \|^2 + h(\uu_{t + 1}) \\
	\\
	& = & 
	r(\uu_t) + \alpha_t \la \nabla r(\uu_t), \vv_{t + 1} - \uu_t \ra
	+ \frac{\beta \alpha_t^2}{2}\|\vv_{t + 1} - \uu_t\|^2 \\
	\\
	& & \qquad \; + \; h(\alpha_t \vv_{t + 1} + (1 - \alpha_t) \uu_t) \\
	\\
	& \leq & 
	s(\uu_t) + \alpha_t \bigl( 
	\la \nabla r(\uu_t), \vv_{t + 1} - \uu_t \ra
	+ h(\vv_{t + 1}) - h(\uu_t)
	\bigr) + \; \frac{\beta \alpha^2\diam^2}{2} \\
	\\
	& \equiv &
	s(\uu_t) - \alpha_t \Delta_t + \frac{\beta \alpha_t^2\diam^2}{2},
	\ea
	\eeq
	where the last equality comes from the definition of $\Delta_t$ in Algorithm~\ref{alg:InexactProx}.
	
	Note that $\alpha_t$ is defined as the minimizer of $\frac{\beta \alpha_t^2}{2} \|\vv_{t + 1} - \uu_t \|^2 - \alpha_t \Delta_t $ and hence, for any other $\rho_t \in [0, 1]$ it will hold that: 
	
	\begin{equation}
	\label{ProofThSubStep2}
	        	s(\uu_{t + 1})  \;\; \leq \;\;s(\uu_t) - \rho_t \Delta_t + \frac{\beta \rho_t^2\diam^2}{2} .
	\end{equation}

	At the same time,
	\beq \label{ProofThSubLB}
	\ba{rcl}
	\Delta_t & \Def & h(\uu_t) - h(\vv_{t+1}) -  \la \nabla r(\uu_t), \vv_{t+1} - \uu_t \ra \\
	\\
	& \overset{\eqref{ProofThSubVt}}{\geq} & 
	h(\uu_t) - h(\uu) - \la \nabla r(\uu_t), \uu - \uu_t \ra \\
	\\
	& \geq &
	s(\uu_t) - s(\uu), \qquad \forall \uu \in \X
	\ea
	\eeq
	where the last line follows from the convexity of $r(\uu)$.
	Letting $\uu := \uu^{\star}$ (solution to \eqref{ProofThSubMain}) in \eqref{ProofThSubLB} and further substituting it into \eqref{ProofThSubStep2} and subtracting $\ss^{\star}$ from both sides, we obtain
	\beq \label{ProofThSubStep3}
	\ba{rcl}
	\left[ s(\uu_{t + 1}) - s^{\star} \right] & \leq & (1 - \rho_t) \left[ s(\uu_t) - \ss^{\star} \right] + \frac{\beta\rho_t^2 \diam^2}{2}.
	\ea
	\eeq
	
	Now, let us choose $\rho_t := \frac{a_{t+1}}{A_{t+1}}$ for sequences $A_t := t \cdot (t + 1)$, and $a_{t + 1} := A_{t + 1} - A_t = 2(t + 1)$. Then,
	$
	\ba{rcl}
	\rho_t & := &  \frac{2}{2 + t}, \;\; t \geq 0.
	\ea
	$
	Using this choice, inequality \eqref{ProofThSubStep3} can be rewritten as
	$$
	\ba{rcl}
	A_{t + 1} \bigl[ s(\uu_{t + 1}) - s^{\star} \bigr] 
	& \leq & 
	A_t \bigl[ s(\uu_t)  - s^{\star} \bigr]
	+ \frac{a_{t + 1}^2\beta\diam^2}{2A_{t+1}}
	\ea
	$$
	Telescoping this inequality for the first iterations, we obtain, for $t \geq 1$:
	\beq \label{ProofThSubResConv}
	\ba{rcl}
	s(\uu_t) - s^{\star} & \leq &  \frac{\beta \diam^2}{2 A_t} \cdot \sum\limits_{i = 1}^t \frac{a_i^2}{A_i}
	\;\; = \;\;
	\frac{\beta \diam^2}{2t(t + 1)} \cdot \sum\limits_{i = 1}^t \frac{4i}{i+1}
	\;\; \leq \;\;
	\frac{2 \beta \diam^2}{t + 1}.
	\ea
	\eeq
	This is the global convergence in terms of the functional residual.
	It remains to justify the convergence for the accuracy certificates $\Delta_t$.
	Multiplying \eqref{ProofThSubStep2} by $A_{t + 1}$, we obtain
	\beq \label{ProofThSubStep4}
	\ba{rcl}
	a_{t + 1} \Delta_t & \leq & a_{t + 1} s(\uu_t) + A_t s(\uu_t) - A_{t + 1} s(\uu_{t + 1})
	+ \frac{a_{t + 1}^2}{A_{t + 1}} \frac{\beta \diam^2}{2}.
	\ea
	\eeq
	Telescoping this bound, we get,
	for $t \geq 1$:
	$$
	\ba{rcl}
	\sum\limits_{i = 1}^t a_{i + 1}
	\cdot \min\limits_{1 \leq i \leq t} \Delta_i
	& \leq & 
	\sum\limits_{i = 1}^t a_{i + 1} \Delta_i \\
	\\
	& \overset{\eqref{ProofThSubStep4}}{\leq} &
	a_1 \bigl[ s(\uu_1) - s^{\star} \bigr]
	+ \sum\limits_{i = 1}^t a_{i + 1} \bigl[ s(\uu_i) - s^{\star}  \bigr]
	+ \frac{\beta \diam^2}{2} \sum\limits_{i = 1}^t \frac{a_{i + 1}^2}{A_{i + 1}} \\
	\\
	& \overset{\eqref{ProofThSubResConv}}{\leq} &
	2 \beta \diam^2
	\cdot \Bigl( 
	1 + \sum\limits_{i = 1}^t \frac{a_{i + 1}}{i + 1}
	+ \frac{1}{4} \sum\limits_{i = 1}^t \frac{a_{i + 1}^2}{A_{i + 1}} 
	\Bigr)  \\
	\\
	& \leq & 
	2 \beta \diam^2 \cdot (1 + 3t).
	\ea
	$$
	Dividing both sides by $\sum_{i = 1}^t a_{i + 1} = A_{t + 1} - A_1 = t(3 + t)$ completes the proof we finally get:
	\begin{align*}
	    \min\limits_{1 \leq i \leq t} \Delta_i \;\; \leq \;\; \frac{6 \beta \diam^2 }{t}.
	\end{align*}
\end{proof}

\subsection{Proof of Proposition \ref{PropNorm}}
\label{app:proof-prop-3-1}
\PropNorm*
\begin{proof}
	In our case, we have $\varphi(\xx) \equiv \| \ff(\xx) \|_2$.
	Using Lemma~\ref{LemGlobalBounds}, we obtain
	\beq \label{GNOneStep}
	\ba{rcl}
	\varphi(\yy_{k + 1}) & 
	\leq &
	\| \ff(\yy_k) + \nabla f(\yy_k) (\yy_{k + 1} - \yy_k) \|_2
	+ \frac{\gamma_k^2}{2} \mathcal{S} \\
	\\
	& = &
	\| \ff(\yy_k) + \gamma_k \nabla f(\yy_k) (\xx_{k + 1} - \yy_k) \|_2
	+ \frac{\gamma_k^2}{2} \mathcal{S},
	\ea
	\eeq
	where $\xx_{k+1} \in \X$ is the point such that $\yy_{k+1} = \yy_k + \gamma_k(\xx_{k+1} - \yy_k)$.
	Using convexity of the function
	$g(\xx) \Def \| \ff(\yy_k) + \gamma_k \nabla \ff(\yy_k)(\xx - \yy_k)\|_2$,
	we get that
	$$
	\ba{rcl}
	\varphi(\yy_k) & = & g(\yy_k)
	\;\; \geq \;\;
	g(\xx_{k + 1}) + \la g'(\xx_{k + 1}), \yy_k - \xx_{k + 1} \ra \\
	\\
	& = &
	\| \ff(\yy_k) + \gamma_k \nabla f(\yy_k) (\xx_{k + 1} - \yy_k) \|_2
	+ \la g'(\xx_{k + 1}), \yy_k - \xx_{k + 1} \ra, 
	\ea
	$$
	where
	the subgradient $g'(\xx_{k + 1}) = \gamma_k \nabla \ff(\yy_k)^{\top} \frac{\ff_{k + 1}}{\| \ff_{k + 1} \|_2}$
	with $\ff_{k + 1} \Def \ff(\yy_k) + \gamma_k\nabla \ff(\yy_k)(\xx_{k + 1} - \yy_k )$,
	satisfies the stationary condition for the method step:
	\beq \label{PropStatCond2}
	\ba{rcl}
	\la g'(\xx_{k + 1}), \xx - \xx_{k + 1} \ra & \geq & 0, \qquad \forall \xx \in \X.
	\ea
	\eeq
	A few comments are in order now about the use of the subgradient above. Note that we wish to impose an assumption on $\ff$ which can ensure that $\ff(\yy_k) + \gamma_k \nabla \ff(\yy_k)(\xx - \yy_k) \neq \mathbf{0} \in \R^n$. First, some preliminaries. Under Assumption~\ref{AssumptionLipCont} on $\ff$, it holds that:
	\begin{align}
		& \exists \mathcal{F} \in (0, \infty) \text{ s.t. } \norm{\ff(x)} \leq \mathcal{F}, \forall \xx \in \X \quad \text{ by continuity of $\ff$} \label{app-eq-finite-func}\\
		& \exists \mathcal{G} \in (0, \infty) \text{ s.t. } \norm{\nabla \ff(x)} \leq \mathcal{G}, \forall \xx \in \X \quad \text{ by continuous differentiability of $\ff$} \label{app-eq-finite-grad}
	\end{align}
	From here, we can bound the products between Jacobians and iterates as follows:
	\begin{equation}
			\norm{ \nabla \ff(\xx) (\yy - \zz)} \leq \norm{ \nabla \ff(\xx)} \norm{\yy - \zz} \leq \mathcal{G} \diam, \quad \forall \xx, \yy, \zz \in \X.
	\end{equation}
	Thus, without loss of generality, we can shift $\ff$ by a constant vector of identical values depending on $\mathcal{G} \diam$ such that we ensure, for example, $\ff(\yy_k) + \gamma_k \nabla \ff(\yy_k)(\xx - \yy_k) > \mathbf{0}$ component-wise.
	Hence, combining these observations with \eqref{GNOneStep}, we have
	$$
	\ba{rcl}
	\varphi(\yy_k) - \varphi(\yy_{k + 1})  & \geq & 
	\la g'(\xx_{k + 1}), \yy_k - \xx_{k + 1} \ra
	- \frac{\gamma_k^2}{2} \mathcal{S}	\\
	\\
	&  \overset{\eqref{PropStatCond2}}{\geq} &
	\max\limits_{\xx \in \X} \la g'(\xx_{k + 1}), \yy_k - \xx \ra
	- \frac{\gamma_k^2}{2} \mathcal{S}.
 	\ea
	$$
Then, by lower bounding appropriately using~\eqref{app-eq-finite-func} and \eqref{app-eq-finite-grad}, we get:
	$$
	\ba{rcl}
	\varphi(\yy_k) - \varphi(\yy_{k + 1}) & \geq  &
	\frac{\gamma_k}{\mathcal{F} + \mathcal{G} \diam}
	\max\limits_{\yy \in \X} \la \nabla \ff(\yy_k)^{\top} \ff(\yy_k), \yy_k - \yy \ra
	- \gamma_k^2\Bigl(  \frac{\mathcal{G} \diam^2}{\mathcal{F} + \mathcal{G} \diam}  
	+ \frac{\mathcal{S}}{2}\Bigr) \\
	\\
	& = &
	\frac{\gamma_k}{\mathcal{F} + \mathcal{G} \diam}
	\max\limits_{\yy \in \X}
	\la \nabla \Phi(\yy_k), \yy_k - \yy \ra
	- \gamma_k^2\Bigl(  \frac{\mathcal{G} \diam^2}{\mathcal{F} + \mathcal{G} \diam}  
	+ \frac{\mathcal{S}}{2}\Bigr).
	\ea
	$$
	Substituting $\gamma_k := \frac{1}{\sqrt{1 + k}}$ and telescoping this bound
	would lead to the desired global convergence (for the details, see the end of the proof
	of Theorem~\ref{TheoremGap}).
\end{proof}

\section{Interpretation of $\Delta_k$ in the non-convex setting}
\label{app-section-interp-delta-nonconvex}

While we cannot make any strong claims about the meaning of $\Delta_k$ in general, we can provide an additional observation for this quantity 
when the outer component $F$ is smooth inside a ball included in $\X$. 

Thus, consider a ball of radius $\varepsilon$ centered at $\yy_k$ denoted by $B(\yy_k, \varepsilon) = \{ \xx \in \R^d \; : \; \| \xx - \yy_k \| \leq \varepsilon  \}$,
and set $\mathcal{B} =  B(\yy_k, \varepsilon) \cap \X$. Assuming that $F(\uu, \xx)$ is differentiable at all points from $\R^n \times \mathcal{B}$,
and that its gradient is Lipschitz continuous with constant $L_F$, we have for any $\xx \in \mathcal{B} \subseteq \mathcal{X}$:
$$
\ba{rcl}
\Delta_k & = &
\max\limits_{\xx \in \X}\Bigl[ \varphi(\yy_k)
-  F\big(\ff(\yy_k) + \nabla \ff(\yy_k)(\xx - \yy_k), \xx\big) \Bigr]\\ 
\\
& \geq & \max\limits_{\xx \in \mathcal{B}}\Bigl[
\varphi(\yy_k)
- F(\ff(\yy_k), \yy_k)
- \la \frac{\partial F}{\partial \uu}( \ff(\yy_k), \yy_k ), \nabla \ff(\yy_k)(\xx - \yy_k) \ra \\
\\
& & \qquad \quad
- \; \la \frac{\partial F}{\partial \xx}( \ff(\yy_k), \yy_k ), \xx - \yy_k \ra
- \frac{L_F}{2} \bigl( \| \nabla \ff(\yy_k) \|^2 + 1 \bigr) \cdot \varepsilon^2 \Bigr]\\
\\
& = & 
\max\limits_{\xx \in \mathcal{B}}\Bigl[ \la \nabla \varphi(\yy_k), \yy_k - \xx \ra \Bigr] - \frac{L_F}{2} \bigl( \| \nabla \ff(\yy_k) \|^2 + 1 \bigr) \cdot \varepsilon^2. 
\ea
$$

Hence, for a small enough ball, $\Delta_k$ is an $\bigO{\epsilon^2}$-approximation of the original FW gap restricted to the considered neighborhood. If, in addition, the composite function $\varphi$ is convex in  $\mathcal{B}$ and there is a local optimum $\xopt \in \mathcal{B}$, then $\Delta_k$ is an $\bigO{\epsilon^2}$-approximation of functional suboptimality.

\section{Additional Application Examples}
\label{app-extra-examples}
\begin{example} \label{ExampleGNM}
We define a generalized nonlinear model as, 
\begin{equation}
	\label{app-eq-gnlm}
F(\uu, \xx) \equiv \sum_{i = 1}^n \phi(u^{(i)}),
\end{equation}
where $\phi : \R \to \R$ is a fixed convex loss function,
and $n$ is the number of data points.
Problem \eqref{MainProblem} then reduces to training a (non-convex) model, for example a neural network, with respect to the constraint set $\X$:
$
\ba{rcl}
\min\limits_{\xx \in \X}
\sum\limits_{i = 1}^m \phi( f_i(\xx) ).
\ea
$

Solving this problem then involves training a linear model
within the basic subroutine \eqref{MainSubproblem}
$\displaystyle
\ba{rcl}
\min\limits_{\xx \in \X}
\sum\limits_{i = 1}^m \phi( \la \aa_i, \xx \ra + b_i ),
\ea
$
which is a convex problem. Amongst the loss functions relevant to Machine Learning, the following are convex and subhomogeneous thus making $F$ in~\eqref{app-eq-gnlm} satisfy Assumption~\ref{AssumptionF}: 
\begin{itemize}
	\item $\ell_1$-regression: $\phi(t) = |t|$
	\item Hinge loss (SVM): $\phi(t) = \max\{0, t\}$
	\item Logistic loss: $\phi(t) = \log(1 + e^t)$
\end{itemize}
\end{example}

\end{document}